\documentclass[12pt]{amsart}
\usepackage{amssymb}
\usepackage{amsbsy}
\usepackage{amscd}

\theoremstyle{plain}
\newtheorem{thm}{Theorem}[section]
\newtheorem{prop}[thm]{Proposition}
\newtheorem{lem}[thm]{Lemma}
\newtheorem{cor}[thm]{Corollary}

\theoremstyle{definition}
\newtheorem{dfn}[thm]{Definition}
\newtheorem{ex}[thm]{Example}
\newtheorem{rem}[thm]{Remark}

\numberwithin{equation}{section}

\newcommand{\G}{\Gamma}
\newcommand{\g}{\gamma}

\newcommand{\go}{\omega}

\newcommand{\gd}{\delta}
\newcommand{\gD}{\Delta}
\newcommand{\sm}{\left(\begin{smallmatrix}}
\newcommand{\esm}{\end{smallmatrix}\right)}
\newcommand{\mm}{\begin{pmatrix}}
\newcommand{\emm}{\end{pmatrix}}

\newcommand{\ov}{\overline}
\newcommand{\wt}{\widetilde}
\newcommand{\wh}{\widehat}
\newcommand{\ve}{\varepsilon}
\newcommand{\ga}{\alpha}
\newcommand{\gb}{\beta}

\newcommand{\mb}{\mathbf}
\newcommand{\mbb}{\mathbb}
\newcommand{\mcal}{\mathcal}

\newcommand{\gs}{\sigma}

\newcommand{\gt}{\theta}
\newcommand{\gT}{\Theta}
\newcommand{\gl}{\lambda}

\newcommand{\pa}{\partial}

\newcommand{\bi}{\binom}

\newcommand{\mF}{\mathcal F}
\newcommand{\mH}{\mathcal H}
\newcommand{\mI}{\mathcal I}
\newcommand{\mJ}{\mathcal J}
\newcommand{\mL}{\mathcal L}

\newcommand{\fJ}{\mathfrak J}
\newcommand{\fK}{\mathfrak K}

\newcommand{\mG}{\mathcal G}

\newcommand{\mQ}{\mathcal Q}

\newcommand{\bC}{\mathbb C}
\newcommand{\bR}{\mathbb R}
\newcommand{\bZ}{\mathbb Z}
\newcommand{\bQ}{\mathbb Q}

\newcommand{\bP}{\mathbb P}

\newcommand{\sdf}{\Psi{\mathcal D} (\mathcal F)}

\newcommand{\md}{\operatorname{\text{$\|$}}}
\newcommand{\fS}{\mathfrak S}
\newcommand{\fL}{\mathfrak L}
\newcommand{\fQ}{\mathfrak Q}
\newcommand{\fT}{\mathfrak T}
\newcommand{\PPi}{{^\pa}\Pi}

\newcommand{\BG}{\boldsymbol \Gamma}

\newcommand{\fj}{\widehat{\mathfrak J}}
\newcommand{\fk}{\widehat{\mathfrak K}}

\newcommand{\sh}{\operatorname{Sh}}
\newcommand{\qs}{\operatorname{\text{${_Q}$}Sh}}
\newcommand{\qt}{\operatorname{\text{${_Q}$}\Theta}}
\newcommand{\dsc}{\operatorname{Disc}}

\begin{document}

\title[ Quasimodular forms]{Quasimodular forms, Jacobi-like forms,\\[.1in]
and pseudodifferential operators\\[.5in]}


\author[YoungJu Choie]{YoungJu Choie$^\ast$}

\address{Department of Mathematics and
PMI(Pohang Mathematical Institute), POSTECH, Pohang 790-784,
Korea}

\email{yjc@postech.ac.kr}

\author[Min Ho Lee]{Min Ho Lee$^\dagger$}

\address{Department of Mathematics, University of Northern Iowa,
Cedar Falls, IA 50614, U.S.A.}

\email{lee@math.uni.edu}

\thanks{$^\ast$Supported in part by NRF20090083909 and NRF2009-0094069}
\thanks{$^\dagger$Supported in part by a PDA award from the University of
  Northern Iowa}

\begin{abstract}
We study various properties of quasimodular forms by using their
connections with Jacobi-like forms and pseudodifferential
operators.  Such connections are made by identifying quasimodular
forms for a discrete subgroup $\G$ of $SL(2, \bR)$ with certain
polynomials over the ring of holomorphic functions of the
Poincar\'e upper half plane that   are $\G$-invariant.  We
consider a surjective map from Jacobi-like forms to quasimodular
forms and prove that it has a right inverse, which may be regarded
as a lifting from quasimodular forms to Jacobi-like forms.  We use
such liftings to study Lie brackets and Rankin-Cohen brackets for
quasimodular forms.  We also  discuss Hecke operators and
construct Shimura isomorphisms and Shintani liftings for
quasimodular forms.
\end{abstract}



\maketitle \pagestyle{plain}


\section{\bf{Introduction}}

Quasimodular forms generalize classical modular forms and were introduced
by Kaneko and Zagier in \cite{KZ95}, where they identified quasimodular
forms with generating functions counting maps of curves of genus $g>1$ to
curves of genus $1$.  More generally, they  often appear as generating
functions of counting functions in various problems such as Hurwitz
enumeration problems, which include not only number theoretic problems but
also those in certain areas of applied mathematics (see e.g.\ \cite{BGHZ},
\cite{EO06},\cite{Eo01}, \cite{LR}, \cite{OP06}, \cite{Ro07}).  Unlike
modular forms, derivatives of quasimodular forms are also quasimodular
forms. The goal of this paper is to study various properties of
quasimodular forms by using their connections with Jacobi-like forms and
pseudodifferential operators.  Such connections are made by identifying
quasimodular forms over a discrete subgroup $\G$ of $SL(2, \bR)$ with
certain polynomials over the ring of holomorphic functions of the
Poincar\'e upper half plane that are invariant under an action of $\G$.

Jacobi-like forms for $\G$ are formal power series, whose coefficients are
holomorphic functions on the Poincar\'e upper half plane $\mH$, invariant
under a certain right action of $\G$ (see \cite{CM97}, \cite{Za94}).  This
invariance property is essentially one of the two conditions that must be
satisfied by a Jacobi form (cf.\ \cite{EZ85}), and it determines relations
among coefficients of the given Jacobi-like form.  Such relations can be
used to express each coefficient of a Jacobi-like form in terms of
derivatives of some modular forms for $\G$.  These modular forms belong to
a sequence of the form $\{ f_w \}_{w \geq 0}$, where $f_w$ is a modular
form of weight $2w +\gl$ for some $\gl \in \mbb Z$.  In fact, sequences of
modular forms for $\G$ of such type are in one-to-one correspondence with
Jacobi-like forms for $\G$ of weight $\gl$.  Another interesting aspect of
Jacobi-like forms is that a Jacobi-like form for $\G$ determines a
$\G$-automorphic pseudodifferential operator, which is a formal Laurent
series in $\pa^{-1}$ having holomorphic functions on $\mH$ as coefficients.
There is a natural right action of $SL(2, \bR)$ on the space of
pseudodifferential operators determined by linear factional transformations
of $\mH$, and $\G$-automorphic pseudodifferential operators are
pseudodifferential operators that are invariant under the action of $\G
\subset SL(2, \bR)$.  We can consider an isomorphism between the space of
formal power series and that of pseudodifferential operators, which
provides a correspondence between Jacobi-like forms for $\G$ and
$\G$-automorphic pseudodifferential operators.  Various topics related to
mutual correspondences among Jacobi-like forms, pseudodifferential
operators, and sequences of modular forms were investigated by Cohen,
Manin, and Zagier (see \cite{CM97} and \cite{Za94}).

Given integers $w$ and $m$ with $m \geq 0$, a quasimodular form for $\G$ of
weight $w$ and depth at most $m$ is a holomorphic function $f$ on $\mH$
associated to holomorphic functions $f_0, f_1, \ldots, f_m$ on $\mH$
satisfying
\[ \frac {1} {(cz+d)^w} f \biggl( \frac {az+b} {cz+d} \biggr) = f_0 (z) +
f_1 (z) \biggl( \frac c{cz+d} \biggr) + \cdots + f_m (z) \biggl( \frac
c{cz+d} \biggr)^m \]
for all $z \in \mH$ and $\sm a&b\\ c&d \esm \in \G$.  Then the
quasimodular form $f$ determines the polynomial
\[ F(z,X) = \sum^m_{r=0} f_r (z) X^r ,\]
which we call a quasimodular polynomial. The condition for $f$ to be a
quasimodular form is equivalent to the condition that the polynomial
$F(z,X)$ to be invariant under a certain right action of $\G$ on the space
of polynomials of degree at most $m$ with holomorphic coefficients.  There
is a natural surjective map from the space of Jacobi-like forms to that of
quasimodular forms, and a similar map exists from pseudodifferential
operators to quasimodular forms.  One of the main results of this paper is
the existence of the right inverse of such a surjecctive map, which may be
regarded as a lifting from quasimodular forms to Jacobi-like forms or
to pseudodifferential operators.

The right action of $SL(2, \bR)$ on the space of polynomials over
holomorphic functions on $\mH$ mentioned above can be used to
define Hecke operators on quasimodular polynomials, which are
compatible with the usual Hecke operators on modular forms.  On
the other hand, the action of $SL(2, \bR)$ on formal power series
described earlier determines Hecke operators on Jacobi-like forms.
We show that the surjective maps from Jacobi-like forms to
quasimodular polynomials are Hecke equivariant.  We also study
various applications of the liftings from quasimodular forms to
Jacobi-like forms. In particular, we introduce a Lie algebra
structure on the space of quasimodular polynomials and construct
Rankin-Cohen brackets for quasimodular forms.

As another application of liftings of quasimodular forms to Jacobi-like
forms, we study the Shimura correspondence for quasimodular forms.  Given
an odd integer $k$, the usual Shimura correspondence associates to each
modular form of weight $\gl = (k-1)/2$ a modular form of weight $2\gl-1$.
On the other hand, there is also a map carrying modular forms in the
opposite direction known as the Shintani lifting.  We introduce the notion
of quasimodular forms and Jacobi-like forms of half integral weight and
construct a quasimodular analog of Shimura isomorphisms and Shintani
liftings.

This paper is organized as follows.  In Section 2 we study formal power
series and polynomials over the ring of holomorphic functions on $\mH$.  We
consider linear maps between such power series and polynomials and
introduce actions of $SL(2, \bR)$ that are compatible with respect to these
maps.  In Section 3 we describe quasimodular forms for $\G$ and some of
their properties.  We introduce quasimodular polynomials, which can be
identified with quasimodular forms.  We also consider $\G$-invariant
formal power series known as Jacobi-like forms and discuss their
connections with quasimodular polynomials.  Section 4 is concerned with
Hecke operators on quasimodular polynomials that are compatible with those
on modular and Jacobi-like forms.  In Section 5 we construct a lifting of a
modular form to a quasimodular form whose leading coefficient coincides
with the given modular form.  The existence of liftings of quasimodular
forms to Jacobi-like forms is one of our main findings and is proved in
Section 6.  The noncommutative multiplication operation on the space of
pseudodifferential operators given by the Leibniz rule determines a natural
Lie algebra structure on the same space.  In Section 7 we use the
correspondence between pseudodifferential operators and Jacobi-like forms
as well as the liftings discussed in Section 6 to construct Lie brackets on
the space of quasimodular polynomials.  As another application of the
liftings, we study Rankin-Cohen brackets for quasimodular forms in Section
8.  We extend the notion of quasimodular and Jacobi-like forms to the case
of half integral weight in Section 9, and investigate the quasimodular analog
of correspondences of Shimura and Shintani in Sections 10 and 11.

\section{\bf{Formal power series and polynomials}} \label{S:pp}

In this section we study formal power series and polynomials whose
coefficients are holomorphic functions on the Poincar\'e upper half plane.
We consider linear maps between such power series and polynomials and
introduce actions of $SL(2, \bR)$ that are compatible with respect to these
maps.

The group $SL(2, \bR)$ acts on the Poincar\'e upper half plane $\mH$ as
usual by linear fractional transformations given by
\[ \g z = \frac {az+b} {cz+d} \]
for all $z \in \mH$ and $\g = \sm a&b\\ c&d \esm \in SL(2,\mbb R)$.  For
the same $z$ and $\g$, we set
\begin{equation} \label{E:kz}%
\fJ (\g, z) = cz+d, \quad \fK (\g,z) = \fJ(\g, z)^{-1} \frac{d}{dz}
  (\fJ(\g, z)) = \frac c{cz+d} ,
\end{equation}
so that we obtain the maps $\fJ, \fK: SL(2, \bR) \times \mH \to \bC$.  As
is well-known, the map $\fJ$ is an automorphy factor satisfying the cocycle
condition
\begin{equation} \label{E:kk}
\fJ (\g \g', z) = \fJ (\g, \g' z) \fJ (\g', z)
\end{equation}
for all $\g, \g' \in SL(2, \bR)$ and $z \in \mH$.  On the other hand, it
can be shown that the other map $\fK$ satisfies
\begin{equation} \label{E:5k}
\fK (\g \g', z) = \fK (\g', z) + \fJ (\g', z)^{-2} \fK (\g, \g' z).
\end{equation}

Let $\mF$ denote the ring of holomorphic functions on $\mH$, and let $\mF
[[X]]$ be the complex algebra of formal power series in $X$ with
coefficients in $\mF$.  Given elements $f \in \mF$, $\Phi (z,X) \in \mF
[[X]]$, $\gl \in \bZ$, and $\g \in SL(2, \bR)$, we set
\begin{equation} \label{E:xt}
(f \mid_\gl \g) (z) = \fJ (\g, z)^{-\gl} f (\g z)
\end{equation}
\begin{equation} \label{E:yt}
(\Phi \mid^J_\gl \g) (z,X) = \fJ (\g, z)^{-\gl} e^{- \fK (\g, z) X} \Phi
(\g z, \fJ (\g, z)^{-2} X)
\end{equation}
for all $z \in \mH$.  If $\g'$ is another element of $SL(2, \bR)$, then
from \eqref{E:kk} and \eqref{E:5k} it can be shown that
\[ f \mid_\gl (\g \g') = (f \mid_\gl \g) \mid_\gl  \g', \quad \Phi
\mid^J_\gl (\g \g') = (\Phi \mid^J_\gl \g) \mid^J_\gl  \g' ;\]
hence the operations $\mid_\gl$ and $\mid^J_\gl$ determine right actions
of $SL(2, \bR)$ on $\mF$ and $\mF [[X]]$, respectively. If $\gd$ is a
nonnegative integer, we set
\[ \mF [[X]]_\gd = X^\gd \mF [[X]], \]
so that an element $\Phi (z,X) \in \mF [[X]]_\gd$ can be written in the
form
\begin{equation} \label{E:tf}
\Phi (z,X) = \sum^\infty_{k =0} \phi_k (z) X^{k+\gd}
\end{equation}
with $\phi_k \in \mF$ for each $k \geq 0$.

Given a nonnegative integer $m$, we now consider the complex algebra
$\mF_m [X]$ of polynomials in $X$ over $\mF$ of degree at most $m$.  If
$\g \in SL(2, \bR)$, $\gl \in \bZ$ and $F (z,X) \in \mF_m [X]$, we set
\begin{equation} \label{E:js}
(F \md_\gl \g) (z, X) = \fJ (\g, z)^{-\gl} F (\g z, \fJ (\g, z)^2 ( X -
\fK (\g, z)))
\end{equation}%
for all $z \in \mH$.

\begin{lem}
If $F (z,X) \in \mF_m [X]$ and $\gl \in \bZ$, then we have
\[ ((F \md_\gl \g) \md_\gl \g' )(z, X) = (F \md_\gl (\g \g')) (z, X) \]
for all $\g, \g' \in SL(2, \bR)$.
\end{lem}

\begin{proof}
Given $\g , \g' \in SL(2, \bR)$, $\gl \in \bZ$ and $F (z,X) \in \mF_m
[X]$, using \eqref{E:js}, we have
\begin{align} \label{E:uk}
((F \md_\gl \g) \md_\gl \g') (z, X) &= \fJ (\g', z)^{-\gl} (F \md_\gl \g)
(\g' z,
\fJ (\g', z)^2 ( X - \fK (\g', z)))\\
&= \fJ (\g', z)^{-\gl} \fJ (\g, \g' z)^{-\gl} F (\g \g' z, X' ), \notag
\end{align}
where%
\[ X' = \fJ (\g, \g' z)^2 \fJ (\g', z)^2 ( X - \fK (\g', z)) -  \fJ (\g,
\g' z)^2 \fK (\g, \g' z) .\]
Thus using \eqref{E:kk} and \eqref{E:5k}, we see that
\begin{align*}
X' &= \fJ (\g \g', z)^2 (X - \fK (\g', z)) - \fJ (\g, \g' z)^2 \fJ (\g',
z)^2 (\fK (\g \g', z) - \fK (\g', z))\\
&= \fJ (\g \g', z)^2 (X - \fK (\g \g', z)).
\end{align*}
From this and \eqref{E:uk}, we obtain%
\begin{align*}
((F \md_\gl \g) \md_\gl \g') (z, X) &= \fJ (\g \g', z)^{-\gl} F (\g \g' z,
\fJ (\g
\g', z)^2 (X - \fK (\g \g', z))\\
&= (F \md_\gl (\g \g')) (z, X);
\end{align*}
hence the lemma follows.
\end{proof}

If $\Phi (z,X) \in \mF [[X]]_\gd$ is as in \eqref{E:tf} and if $m$
is a nonnegative integer, we set
\begin{equation} \label{E:mc}
(\Pi^\gd_m \Phi) (z, X) = \sum^m_{r =0} \frac{1}{r!} \phi_{m-r} (z) X^r,
\end{equation}
which determines the surjective complex linear map
\[ \Pi^\gd_m: \mF [[X]]_\gd \to \mF_m [X] .\]
In particular, for $m=0$ we obtain the map
\[ \Pi_0^\gd: \mF [[X]]_\gd \to \mF \]
given by
\[ (\Pi_0^\gd \Phi) (z) = \phi_0 (z) .\]
We also define the linear maps
\begin{equation} \label{E:v6}
\quad \fS_\ell: \mF_m [X] \to \mF
\end{equation}
for $0 \leq \ell \leq m$ by setting
\begin{equation} \label{E:sy}
(\fS_\ell F) (z) = f_\ell (z)
\end{equation}
for $\Phi (z,X) \in \mF [[X]]_\gd$ as above and $F (z,X) = \sum_{k=0}^m
f_k (z) X^k \in \mF_m [X]$.

\begin{lem}
Given $m, \gd \geq 0$, the diagram
\begin{equation} \label{E:dd}
\begin{CD}
0 @>>> \mF [[X]]_{\gd +1} @> \wh{\iota} >> \mF [[X]]_\gd @> \wh{\fS}^\gd
>>
\mF @>>> 0\\
@. @VV \Pi^{\gd+1}_{m -1} V @VV \Pi^\gd_m V @VV \mu_m V @.\\
0 @>>> \mF_{m-1} [X] @> \iota >> \mF_m [X] @> \fS_m >> \mF @>>> 0
\end{CD}
\end{equation}
commutes, where $\wh{\iota}$ and $\iota$ are inclusion maps and $\mu_m$ is
multiplication by $(1/m!)$. Furthermore, the two rows in the diagram are
short exact sequences.
\end{lem}

\begin{proof}
Let $\Phi (z,X) \in \mF [[X]]_\gd$  and $\Psi (z,X) \in \mF [[X]]_{\gd
+1}$ be given by
\[ \Phi (z,X) = \sum_{k=0}^{\infty} \phi_k(z) X^{k +\delta}, \quad
\Psi (z,X) = \sum_{k=0}^{\infty} \psi_k(z) X^{k +\delta +1} .\]
Then from \eqref{E:mc} and \eqref{E:sy} we see that
\[ ((\fS_m \circ \Pi^\gd_m) \Phi) (z) = \phi_0 (z) / m! = (\mu_m \phi_0)
(z) .\]
Since $\phi_0 = \wh{\fS}^\gd \Phi$ by \eqref{E:sy}, we obtain
\[ \fS_m \circ \Pi^\gd_m = \mu_m \circ \wh{\fS}^\gd .\]
On the other hand, we have
\[ (\wh{\iota} \Psi) (z,X) = \sum^\infty_{k=0} \wh{\psi}_k (z) X^{k + \gd}
\]
with $\wh{\psi}_0 =0$ and $\wh{\psi}_k = \psi_{k-1}$ for $1 \leq k \leq
m$.  Hence we see that
\begin{align*}
((\Pi^\gd_m \circ \wh{\iota}) \Psi) (z, X) &= \sum^m_{r =0} \frac{1}{r!}
\wh{\psi}_{m-r} (z) X^r = \sum^{m -1}_{r =0} \frac{1}{r!} \psi_{m-1 -r}
(z) X^r\\
&= (\Pi^{\gd+1}_{m -1} \Psi) (z, X) = ((\iota \circ \Pi^{\gd+1}_{m -1})
\Psi) (z, X) ,
\end{align*}
which shows that $\Pi^\gd_m \circ \wh{\iota} = \iota \circ \Pi^{\gd+1}_{m
-1}$.  The two rows in the diagram are clearly short exact sequences.
\end{proof}

The next proposition shows that the surjective map $\Pi^\gd_m: \mF
[[X]]_\gd \to \mF_m [X]$ is equivariant with respect to the right
actions of $SL(2, \bR)$ of the type described earlier.

\begin{prop} \label{P:nk}
If $m$, $\gd$ and $\gl$ are integers with $m, \gd \geq 0$, we have
\begin{equation} \label{E:pv}
\Pi^\gd_m (\Phi \mid^J_\gl \g) = \Pi^\gd_m (\Phi )\md_{\gl +2m +2\gd} \g
\end{equation}
for all $\Phi (z,X) \in \mF [[X]]_\gd$ and $\g \in SL(2, \bR)$.
\end{prop}

\begin{proof}
Let $\Phi (z,X) \in \mF [[X]]_\gd$ be as in \eqref{E:tf}, and let $\g \in
SL(2, \bR)$.  Then, using \eqref{E:yt}, we have
\begin{align*}
(\Phi \mid^J_\gl \g) (z, X) &= \sum^\infty_{k =0} \fJ (\g, z)^{-\gl}
\sum^\infty_{\ell =0} \frac {(-1)^\ell} {\ell !} \fK (\g, z)^\ell X^\ell
\phi_k (\g z) \fJ (\g, z)^{-2k -2\gd} X^{k +\gd}\\
&= \sum^\infty_{k =0} \sum^\infty_{\ell =0} \frac {(-1)^\ell} {\ell !} \fJ
(\g, z)^{-2k -2\gd -\gl} \fK (\g, z)^\ell \phi_k (\g z) X^{k +\ell +\gd}\\
&= \sum^\infty_{n =0} \wt{\phi}_n (z, X) X^{n +\gd}
\end{align*}
with
\[ \wt{\phi}_n (z, X) = \sum^n_{\ell =0} \frac {(-1)^\ell} {\ell !} \fJ
(\g, z)^{-2n +2\ell -2\gd -\gl} \fK (\g, z)^\ell \phi_{n -\ell}
(\g z) .\] From this and \eqref{E:mc} we see that
\begin{align} \label{E:yw}
(\Pi^\gd_m (\Phi \mid^J_\gl \g)) (z, X) &= \sum^m_{r =0} \frac 1{r!}
\wt{\phi}_{m-r} (z) X^r\\
&= \sum^m_{r =0} \sum^{m-r}_{\ell =0} \frac {(-1)^\ell} {r! \ell !} \fJ
(\g, z)^{2\ell +2r -2m -2\gd -\gl} \fK (\g, z)^\ell \phi_{m -\ell -r} (\g
z) X^r. \notag
\end{align}
On the other hand, using \eqref{E:js} and \eqref{E:mc}, we have
\begin{align*}
((\Pi^\gd_m \Phi) \md_{\gl +2m +2\gd} \g) (z, X) &= \fJ (\g, z)^{-\gl -2m
-2\gd} (\Pi^\gd_m \Phi) (\g z, \fJ (\g, z)^2 ( X - \fK (\g, z)))\\
&= \fJ (\g, z)^{-\gl -2m -2\gd} \sum^m_{j =0} \frac 1{j !} \phi_{m
-j} (\g z) \fJ (\g, z)^{2j} (X - \fK (\g, z))^j\\
&= \sum^m_{j =0} \sum^j_{r =0} \frac 1{j !} \bi {j}r
\phi_{m -j} (\g z) \fJ (\g, z)^{2j -\gl -2m -2\gd}\\
&\hspace{1.8in} \times (-1)^{j -r} \fK (\g, z)^{j -r} X^r\\
&= \sum^m_{r =0} \sum^m_{j =r} \frac {(-1)^{j -r}}{r! (j -r) !} \fJ (\g,
z)^{2j -\gl -2m -2\gd} \fK (\g, z)^{j -r} \phi_{m -j} (\g z) X^r.
\end{align*}
Changing the index for the second summation in the previous line from $j$
to $\ell = j-r$ and comparing this with \eqref{E:yw}, we obtain the
relation \eqref{E:pv}.
\end{proof}

\section{\bf{Quasimodular forms}} \label{S:qf}

Jacobi-like forms for $\G$ are formal power series, whose coefficients are
holomorphic functions on the Poincar\'e upper half plane $\mH$, invariant
under a certain right action of $\G$ (see \cite{CM97}, \cite{Za94}).

In this section we describe quasimodular forms for a discrete
subgroup $\G$ of $SL(2, \bR)$ introduced by Kaneko and Zagier (see
\cite{KZ95}) and study some of their properties.  We introduce
quasimodular polynomials, which are polynomials with holomorphic
coefficients invariant under a right action of $\G$ and can be
identified with quasimodular forms.  We also consider
$\G$-invariant formal power series known as Jacobi-like forms
(cf.\ \cite{CM97}, \cite{Za94})  and discuss their connections
with quasimodular polynomials.

Let $\mF$ be the ring of holomorphic functions on $\mH$
as in Section \ref{S:pp}, and let $\G$ be a discrete subgroup of
$SL(2, \bR)$.  We also fix a nonnegative integer $m$.

\begin{dfn} \label{D:mq}
Let $\xi$ be an integer, and let $\mid_\xi$ be the operation in \eqref{E:xt}.

(i) An element $f \in \mF$ is a {\em modular form for $\G$ of weight $\xi$\/}
if it satisfies
\[
f \mid_\xi \g = f
\]
for all $\g \in \G$.  We denote by $M_\xi (\G)$ the space of modular forms
for $\G$ of weight $\xi$.

(ii) An element $f \in \mF$ is a {\em quasimodular form for $\G$ of weight
  $\xi$ and depth at most $m$\/} if there are functions $f_0, \ldots, f_m \in
\mF$ such that
\begin{equation} \label{E:qq1}
(f \mid_\xi \g) (z) = \sum^m_{r=0} f_r (z) \fK (\g, z)^r
\end{equation}
for all $z \in \mH$ and $\g \in \G$, where $\fK (\g, z)$ is as in
\eqref{E:kz}.  We denote by $QM^m_\xi (\G)$ the space of quasimodular forms
for $\G$ of weight $\xi$ and depth at most $m$.
\end{dfn}

\begin{rem} \label{R:gw}
(i) Although the usual definition of modular forms or quasimodular forms
includes a cusp condition, we have suppressed such conditions in Definition
\ref{D:mq}.

(ii) If we set $\g \in \G$ in \eqref{E:qq1} to be the identity
matrix, then $\fK (\g, z) =0$; hence it follows that
\[ f = f_0 .\]
If $m=0$, from \eqref{E:qq1} we obtain
\[ f \mid_\xi \g = f_0 = f ,\]
and therefore we have
\[ QM^0_\xi (\G) = M_\xi (\G) .\]

(iii) For fixed $z \in \mH$, by considering the right hand side of
\eqref{E:qq1} as a polynomial in $\fK (\g, z)$ and using the fact
that it is valid for all $\g$ in the infinite set $\G$, we see
that $f$ determines the coefficients $f_0, \ldots, f_m$ uniquely.
\end{rem}

Let $f \in \mF$ be a quasimodular form belonging to $QM^m_\xi
(\G)$ satisfying \eqref{E:qq1}.  Then we define the corresponding
polynomial $(\mQ_\xi^m f) (z,X) \in \mF_m [X]$ by
\begin{equation} \label{E:tp}
(\mQ_\xi^m f) (z,X) = \sum^m_{r=0} f_r (z) X^r
\end{equation}
for $z \in \mH$.  We note that $\mQ_\xi^m f$ is well-defined due to Remark
\ref{R:gw}(iii).  Thus we obtain the linear map
\[ \mQ_\xi^m: QM^m_\xi (\G) \to \mF_m [X] \]
for each $\xi \in \bZ$.

\begin{dfn} \label{D:8d}
A {\em quasimodular polynomial for $\G$ of weight $\xi$ and degree at most
  $m$\/} is an element of $\mF_m [X]$ that is invariant with respect to the
right $\Gamma$-action in \eqref{E:js}.  We denote by
\[ QP^m_{\xi} (\G) =\{ F (z,X)\in \mF_m [X]  \mid  F\md_{\xi}\g =F
\; \text{for all} \; \g \in \G \} \]
the space of all quasimodular polynomials for $\G$ of weight $\xi$ and degree
at most $m$, where $\md_\xi$ is as in \eqref{E:js}.
\end{dfn}

\begin{lem}
Given $F(z,X) = \sum^m_{r=0} f_r (z) X^r \in \mF_m [X]$ and $\g \in SL(2,
\bR)$, we have
\begin{equation} \label{E:rf1}
\fS_r(F \md_{\xi}\g) (z)= \sum^m_{\ell =r} (-1)^{\ell-r} \bi {\ell}r f_\ell
(\g z) \fJ (\g, z)^{2\ell -\xi} \fK (\g,z)^{\ell-r} ,
\end{equation}
\begin{equation} \label{E:gw}
(\fS_r (F \md_{\xi} \g^{-1}) \mid_{\xi -2r} \g) (z) = \sum^m_{\ell = r} \bi
{\ell}r f_\ell (z) \fK (\g,z)^{\ell -r}
\end{equation}
for $0 \leq r \leq m$.
\end{lem}

\begin{proof}
Using \eqref{E:js}, we have
\begin{align*}
(F \md_\xi \g) (z,X) &= \fJ (\g,z)^{-\xi} \sum^m_{\ell=0} f_\ell
(\g z) \fJ (\g,
z)^{2\ell} (X - \fK (\g,z))^\ell\\
&= \sum^m_{\ell=0} \sum^\ell_{r=0} \bi {\ell}r f_\ell (\g z) \fJ (\g,
z)^{2\ell -\xi}
(-1)^{\ell-r} \fK (\g,z)^{\ell-r} X^r\\
&= \sum^m_{r=0} \sum^m_{\ell=r} (-1)^{\ell-r} \bi {\ell}r f_\ell (\g z)
\fJ (\g, z)^{2\ell -\xi} \fK (\g,z)^{\ell-r} X^r ,
\end{align*}
which verifies \eqref{E:rf1}.  Replacing $\g$ by $\g^{-1}$ and $z$ by $\g
z$ in \eqref{E:rf1}, we have
\[ \fS_r (F \md_{\xi} \g^{-1}) (\g z) = \fJ (\g^{-1}, \g z)^{-\xi} \sum^m_{\ell =r} (-1)^{\ell-r}
\bi {\ell}r f_\ell (z) \fJ (\g^{-1}, \g z)^{2\ell}
\fK (\g^{-1}, \g z)^{\ell-r} .\]
However, from \eqref{E:kk} and \eqref{E:5k} and the fact that $\fJ (1,z)
=1$ and $\fK (1,z) =0$, we see that
\begin{equation} \label{E:jp}
\fJ (\g^{-1}, \g z) = \fJ (\g, z)^{-1}, \quad \fK (\g^{-1}, \g z) = - \fJ
(\g, z)^{2} \fK (\g, z) ,
\end{equation}
and therefore it follows that
\begin{align*}
\fS_r (F \md_{\xi} \g^{-1}) (\g z) &= \fJ (\g, z)^{\xi} \sum^m_{\ell =r}
(-1)^{\ell-r} \bi {\ell}r f_\ell (z) \fJ (\g,
z)^{-2\ell} (-1)^{\ell-r} \fJ (\g, z)^{2\ell -2r} \fK (\g, z)^{\ell-r}\\
&= \fJ (\g, z)^{\xi -2r} \sum^m_{\ell =r} \bi {\ell}r f_\ell (z) \fK (\g,
z)^{\ell-r} ;
\end{align*}
hence we obtain \eqref{E:gw}.
\end{proof}

By setting $r=0$ in \eqref{E:rf1} we obtain
\begin{equation} \label{E:44}
\fS_0 ( F \md_\xi \g) (z) =  \sum^m_{\ell=0} (-1)^\ell \fJ (\g, z)^{2\ell
  -\xi} \fK (\g, z)^\ell f_\ell (\g z) ,
\end{equation}
\[ \fS_0 ( F \md_\xi \g^{-1}) (\g z) = \fJ (\g,z)^\xi \sum^m_{r=0} f_r (z)
\fK (\g, z)^r , \]
where we used \eqref{E:jp}. In particular, if $F(z, X) \in QP^m_\xi (\G)$,
we have
\[ (\fS_0  F) (z) =  \sum^m_{\ell=0} (-1)^\ell \fJ (\g, z)^{2\ell -\xi}
\fK (\g, z)^\ell f_\ell (\g z) \]
for all $\g \in \G$.

\begin{cor} \label{C:mh}
Let $F (z,X) = \sum^m_{r=0} f_r (z) X^r \in \mF_m [X]$.
Then $F(z,X)$ is a quasimodular polynomial belonging to $QP^m_{\xi} (\G)$
if and only if for each $r \in \{ 0, 1, \ldots, m \}$ the coefficient $f_r$
satisfies
\begin{equation} \label{E:ep}
(f_r \mid_{\xi -2r} \g) (z) = \sum^m_{\ell = r} \bi {\ell}r f_\ell (z) \fK
(\g,z)^{\ell -r} = \sum^{m-r}_{\ell = 0} \bi {\ell +r}r f_{\ell +r} (z)
\fK (\g,z)^\ell
\end{equation}
for all $z \in \mH$ and $\g \in \G$.  In particular, $f_r$ is a
quasimodular form belonging to $QM^{m-r}_{\xi -2r} (\G)$.
\end{cor}

\begin{proof}
A polynomial $F (z,X) \in \mF_m [X]$ belongs to $QP^m_{\xi} (\G)$ if an
only if
\[ F (z,X) = (F \md_\xi \g) (z,X) \]
for all $\g \in \G$.  Hence \eqref{E:ep} follows from this and
\eqref{E:gw}.
\end{proof}

If $0 \leq \ell \leq m$, we obtain the complex linear map
\[ \fS_\ell: QP^m_\xi (\G) \to \mF \]
by restricting $\fS_\ell$ in \eqref{E:v6} to $QP^m_\xi (\G)$.  Then $\fS_r
F$ belongs to $QM^{m-r}_{\xi -2r} (\G)$ for $0 \leq r \leq m$, and
\eqref{E:ep} can be written in the form
\[ (\fS_r F \mid_{\xi -2r} \g) (z) = \sum^{m-r}_{\ell = 0} \bi {\ell +r}r
(\fS_{\ell +r} F) (z) \fK (\g,z)^\ell \]
for $F(z,X) \in QP^m_{\xi} (\G)$.  In particular, we obtain
\[ \fS_m F \mid_{\xi -2m} \g = \fS_m F \]
for all $\g \in \G$.  Thus it follows that
\begin{equation} \label{E:hh}
\fS_m F \in M_{\xi -2m} (\G)
\end{equation}
if $F (z,X) \in QP^m_{\xi} (\G)$.  On the other hand, we also have
\[ (\mQ^{m-r}_{\xi -2r} (\fS_r F)) (z,X) = \sum^{m-r}_{\ell = 0}
\bi {\ell +r}r (\fS_{\ell +r} F) (z) X^\ell \in QP^{m-r}_{\xi -2r} (\G) ,\]
where $\mQ^{m-r}_{\xi -2r}$ is as in \eqref{E:tp}.  Thus, in particular,
we see that the map $\mQ^m_\xi$ given by \eqref{E:tp} determines the
complex linear map
\begin{equation} \label{E:sp}
\mQ_\xi^m: QM^m_\xi (\G) \to QP^m_{\xi} (\G) ,
\end{equation}
so that $\mQ_\xi^m$ carries quasimodular forms to quasimodular
polynomials.  In fact, the the next proposition shows that it is an
isomorphism.

\begin{prop} \label{P:kw}
The restriction of the map $\fS_\ell$ in \eqref{E:v6} with $\ell =0$ to
$QP^m_{\xi} (\G) $ determines an isomorphism%
\begin{equation} \label{E:37}
\fS_0: QP^m_{\xi} (\G) \to QM^m_\xi (\G)
\end{equation}
whose inverse is the map $\mQ^m_\xi$ in \eqref{E:sp}.
\end{prop}

\begin{proof}
If $F(z, X) = \sum^m_{r=0} f_r (z) X^r$ is an element of $QP^m_{\xi} (\G)$,
then we have
\[ (\fS_0 F) (z) = f_0 (z) \]
for $z \in \mH$.  However, from Corollary \ref{C:mh} we see that
\[ (f_0 \mid_{\xi} \g) (z) = \sum^m_{\ell = 0} f_\ell (z) \fK
(\g,z)^{\ell} \]
for all $z \in \mH$ and $\g \in \G$.  Thus we have
\[ ((\mQ^m_\xi \circ \fS_0) F) (z, X) = (\mQ^m_\xi f_0) (z,X)
= \sum^m_{j=0} f_j (z) X^j= F(z, X) .\]
We now consider an element $h \in QM^m_\xi (\G)$ satisfying
\[ (h \mid_{\xi} \g) (z) = \sum^m_{\ell = 0} h_\ell (z) \fK
(\g,z)^{\ell} \]
for all $z \in \mH$ and $\g \in \G$, so that
\[ (\mQ^m_\xi h) (z, X) = \sum^m_{\ell = 0} h_\ell (z) X^\ell .\]
Then we see that
\[ (\fS_0 (\mQ^m_\xi h)) (z) = h_0 (z) = h (z) \]
for all $z \in \mH$, and therefore the proposition follows.
\end{proof}

\begin{cor}
Given $F(z,X)= \sum_{r=0}^{m} f_r(z) X^r \in QP^m_{\gl +2m +2\gd} (\G)$,
we set
\begin{equation} \label{E:cm}
\phi_i = (m-i)! f_{m-i}
\end{equation}
for $0 \leq i \leq m$.  Then for each $k \in \{
0, 1, \ldots, m\}$ we have
\begin{equation} \label{E:bz}
(\phi_k \mid_{2k + 2\gd +\gl} \g) (z) = \sum_{r=0}^{k} \frac{1}{r!}
\fK(\g,z)^r \phi_{k-r} (z)
\end{equation}
for all $z \in \mH$ and $\g \in \G$.
\end{cor}

\begin{proof}
If $0 \leq k \leq m$, using \eqref{E:ep} and \eqref{E:cm}, we have
\begin{align*}
(\phi_k \mid_{2k + 2\gd +\gl} \g) (z) &= (m-k)! (f_{m-k} \mid_{2k + 2\gd
+\gl} \g) (z)\\
&= (m-k)! \sum^m_{\ell = m-k} \bi {\ell} {m-k} f_\ell (z) \fK (\g,
z)^{\ell -m +k}\\
&= \sum^m_{\ell = m-k} \frac {\phi_{m-\ell} (z)} {(\ell -m +k)!} \fK (\g,
z)^{\ell -m +k}
\end{align*}
for all $z \in \mH$.  Hence \eqref{E:bz} can be obtained by changing the
index from $\ell$ to $r = \ell -m +k$ in the previous sum.
\end{proof}

\begin{rem}
Let $f \in QM^m_\xi (\G)$ be a quasimodular form satisfying
\eqref{E:qq1}. In the proof of Theorem 1 in \cite{OA07}, it was
indicated that the corresponding polynomial $F (z,X) = (\mQ_\xi^m
f) (z,X) \in \mF_m [X]$ satisfies
\begin{equation} \label{E:bd}
\fJ (\g, z)^{-\xi} F (\g z, X) = F (z, \fJ (\g, z)^{-2} X + \fK (\g, z))
\end{equation}
for all $z \in \mH$ and $\g \in \G$.  It can be shown, however, that
\eqref{E:bd} is equivalent to the condition that
\[ F \md_\xi \g = F \]
for all $\g \in \G$.  Indeed, if $F (z, X) \in \mF_m [X]$ satisfies
\eqref{E:bd}, by replacing $\g$ and $z$ by $\g^{-1}$ and $\g z$,
respectively, we have
\begin{equation} \label{E:wn}
\fJ (\g^{-1}, \g z)^{-\xi} F (z, X) = F (\g z, \fJ (\g^{-1}, \g z)^{-2} X
+ \fK (\g^{-1}, \g z)).
\end{equation}
From \eqref{E:jp} and \eqref{E:wn} we obtain
\[ F (z, X) = \fJ (\g, z)^{-\xi} F (\g z, \fJ (\g, z)^{2} X -
\fJ (\g, z)^{2} \fK (\g, z)) = (F \md_\xi \g) (z, X) ,\]
where we used \eqref{E:js}.
\end{rem}

We now consider Jacobi-like forms studied by Cohen, Manin and Zagier in
\cite{CM97} and \cite{Za94}.  It turns out that they are closely linked to
quasimodular polynomials and therefore to quasimodular forms.

\begin{dfn} \label{D:wk}
Given an integer $\gl$, a formal power series $\Phi (z,X)$ belonging to
$\mF [[X]]$ is a {\it Jacobi-like form for $\G$ of weight $\gl$\/} if it
satisfies
\[
(\Phi \mid^J_\gl \g) (z,X) = \Phi (z,X)
\]
for all $z \in \mH$ and $\g \in \G$, where $\mid^J_\gl$ is as in
\eqref{E:yt}.
\end{dfn}

We denote by $\mJ_\gl (\G)$ the space of all Jacobi-like forms for $\G$ of
weight $\gl$, and set
\[ \mJ_\gl (\G)_\gd = \mJ_\gl (\G) \cap \mF [[X]]_\gd \]
for each nonnegative integer $\gd$.

\begin{prop} \label{P:p4}
The map $\Pi^\gd_m$ in \eqref{E:mc} induces the complex linear map
\begin{equation} \label{E:7g}
\Pi^\gd_m: \mJ_\gl (\G)_\gd \to QP^m_{\gl +2m +2\gd} (\G)
\end{equation}
for each $\gd \geq 0$ and $\gl \in \bZ$.
\end{prop}

\begin{proof}
Given $\gd, \gl \in \bZ$ with $\gd \geq 0$, we need to show that
\[ \Pi^\gd_m (\mJ_\gl (\G)_\gd) \subset QP^m_{\gl +2m +2\gd} (\G) .\]
However, this follows immediately from Proposition \ref{P:nk}, Definition
\ref{D:8d} and Definition \ref{D:wk}.
\end{proof}

\section{\bf{Hecke operators}}

Hecke operators acting on the space of modular forms play an important
role in number theory, and similar operators can also be defined on the
space of Jacobi-like forms.  In this section we introduce Hecke operators
on quasimodular polynomials that are compatible with those on modular and
Jacobi-like forms.

Let $GL^+ (2, \mbb R)$ be the group of $2\times 2$ real matrices of
positive determinant, which acts on the Poincar\'e upper half plane $\mH$
by linear fractional transformations.  We extend the functions $\fJ$ and
$\fK$ given by \eqref{E:kz} to the maps $\fJ, \fK: GL^+ (2, \mbb R) \times
\mH \to \bC$ by setting
\begin{equation} \label{E:ww}
\fJ (\ga, z) = cz+d, \quad \fK (\ga, z) = \fJ(\ga,z)^{-1} \frac{d}{dz}
(\fJ(\ga,z))
\end{equation}
for $z \in \mH$ and $\ga = \sm a&b\\ c&d \esm \in GL^+ (2, \mbb R)$.
Then it can be shown that $\fJ$ satisfies the cocycle condition
\[
\fJ (\ga \ga', z) = \fJ (\ga, \ga' z) \fJ (\ga', z)
\]
for all $z \in \mH$ and $\ga, \ga' \in GL^+ (2, \mbb R)$ as in
\eqref{E:kk}.  On the other hand, the relation \eqref{E:5k} should be
modified as
\[
\fK (\ga \ga', z) = \fK (\ga', z) + (\det \ga') \fJ (\ga', z)^{-2}
\fK (\ga, \ga'z) .
\]

Two subgroups $\G_1$ and $\G_2$ of $GL^+ (2, \mbb R)$ are said to be {\it
commensurable\/} if $\G_1 \cap \G_2$ has finite index in both $\G_1$ and
$\G_2$, in which case we write $\G_1 \sim \G_2$. Given a subgroup $\gD$ of
$GL^+ (2, \mbb R)$, its commensurator $\wt{\gD} \subset GL^+ (2, \mbb R)$
is given by
\[ \wt{\gD} = \{ g \in GL^+ (2, \mbb R) \mid g \gD g^{-1} \sim \gD \}
.\]
If $\G \subset SL(2,\mbb R)$ is a discrete subgroup, the double coset $\G
\ga \G$ with $\ga \in \wt{\G}$ has a decomposition of the form
\begin{equation} \label{E:xm}
\G \ga \G = \coprod_{i=1}^s \G \ga_i
\end{equation}
for some $\ga_i \in GL^+ (2, \mbb R)$ with $i =1, \ldots, s$ (see e.g.\
\cite{Mi89}).

Given an integer $\gl$, we extend the actions of $SL(2, \bR)$ in
\eqref{E:xt} and \eqref{E:js} to those of $GL^+ (2, \bR)$ by
setting
\[
(f \mid_\gl \ga) (z) = \det (\ga)^{\gl/2} \fJ (\ga, z)^{-\gl} f
(\ga z)
\]
\begin{equation} \label{E:ue}
(F \md_\gl \ga) (z, X) = \det (\ga)^{\gl/2} \fJ (\ga, z)^{-\gl} F (\ga z,
\det (\ga)^{-1} \fJ (\ga, z)^2 ( X - \fK (\ga, z)))
\end{equation}%
for all $z \in \mH$, $\ga \in GL^+ (2,\mbb R)$, $f \in \mF$ and $F (z,X)
\in \mF_m [X]$.  These formulas determine right actions of $GL^+ (2,\mbb
R)$ on $\mF$ and $\mF_m [X]$, so that we have
\[
(f \mid_\gl \ga) \mid_\gl \ga' = f \mid_\gl (\ga\ga'),
\]
\begin{equation} \label{E:4hp}
(F \md_\gl \ga) \md_\gl \ga' = F \md_\gl (\ga\ga')
\end{equation}
for all $\ga, \ga' \in GL^+ (2,\mbb R)$.

Let $\ga \in \wt{\G}$ be an element such that the corresponding double
coset is as in \eqref{E:xm}. Then the associated Hecke operator
\begin{equation} \label{E:fs}
T_\gl (\ga): M_\gl (\G) \to M_\gl (\G)
\end{equation}
on modular forms is given as usual by
\begin{equation} \label{E:eb}
(T_\gl (\ga) f) (z) = \sum_{i=1}^s (f \mid_\gl \ga_i) (z)
\end{equation}
for all $f \in M_\gl (\G)$ and $z \in \mH$ (cf.\ \cite{Mi89}).  Similarly,
given a quasimodular polynomial $F (z,X) \in QP^m_{k} (\G)$, we set
\begin{equation} \label{E:ju}
(T^P_k (\ga) F) (z,X) = \sum_{i=1}^s (F \md_k \ga_i) (z, X)
\end{equation}
for all $z \in \mH$.

\begin{prop}
For each $\ga \in \wt{\G}$ the polynomial given by \eqref{E:ju} is
independent of the choice of the coset representatives $\ga_1, \ldots,
\ga_s$, and the map $F \mapsto T^P_\gl (\ga) F$ determines the linear
endomorphism
\[ T^P_\gl (\ga): QP^m_{\gl} (\G) \to QP^m_{\gl} (\G) .\]
\end{prop}

\begin{proof}
Suppose that $\{ \gb_1, \ldots, \gb_s \}$ is another set of coset
representatives with $\gb_i = \g_i \ga_i$ and $\g_i \in \G$ for $1
\leq i \leq s$.  Given $F (z, X) \in QP^m_{\gl} (\G)$, using \eqref{E:4hp},
we have
\begin{align} \label{E:3rp}
\sum_{i=1}^s (F \md_\gl \gb_i) (z, X) &= \sum_{i=1}^s ((F
\md_\gl \g_i)\md_\gl \ga_i) (z, X)\\
&= \sum_{i=1}^s (F \md_\gl \ga_i) (z, X) ,\notag
\end{align}
and hence $T_\gl (\ga) F$ is independent of the choice of the
coset representatives.  Since the linearity of the map $F
\mapsto T_\gl (\ga) F$ is clear, it suffices to show that
$T_\gl (\ga) (QP^m_{\gl} (\G)) \subset QP^m_{\gl} (\G)$.  From \eqref{E:4hp},
\eqref{E:ju} and \eqref{E:3rp} we see that
\[ (T_\gl (\ga) F) \md_\gl \g = \sum_{i=1}^s (F \md_\gl \ga_i)
\md_\gl \g = \sum_{i=1}^s F \md_\gl (\ga_i \g). \]
However, the set $\{ \ga_1 \g, \ldots, \ga_s \g \}$ is another complete
set of coset representatives, and therefore we have
\[ \sum_{i=1}^s F \md_\gl (\ga_i \g) = \sum_{i=1}^s F \md_\gl
\ga_i .\]
Thus we see that $(T_\gl (\ga) F) \md_\gl \g = (T_\gl (\ga)
F)$; the proposition follows.
\end{proof}

We now extend the map
\[ F (z,X) \mapsto \fS_0(F||_{\xi} \g) (z): \mF_m [X] \to \mF \]
for $\g \in SL(2, \bR)$ to $GL^+ (2, \bR)$ by changing \eqref{E:44} to
\[
\fS_0(F||_{\xi} \ga) (z) = \sum^m_{j=0} (-1)^j (\det \ga)^{\xi/2
-j} \fJ (\ga, z)^{2j -\xi} \fK (\ga, z)^j f_j (\ga z)
\]
for $\ga \in GL^+ (2, \bR)$.  If $f$ is a quasimodular form belonging to
$QM^m_\xi (\G)$ and if $\ga \in \wt{\G}$ is an element whose double coset
has a decomposition as in \eqref{E:xm}, we set
\begin{equation} \label{E:ox}
T^Q_\xi (\ga) f = \sum^s_{i=1} \fS_0((\mQ_\xi^m f)||_{\xi}\ga) .
\end{equation}
In other words, $T^Q_\xi (\ga)$ is deduced from $T^P_{\xi}(\ga)$
by conjugation, that is,
\begin{equation}\label{E:rf2}
 T^Q_\xi
(\ga)=\fS_0 T^P_{\xi}(\ga) Q_{\xi}^m. \end{equation}
So this
justifies that $T^Q_\xi (\ga) f \in QM^m_\xi (\G)$. Thus we obtain
the linear endomorphism
\begin{equation} \label{E:e5}
T^Q_\xi (\ga): QM^m_\xi (\G) \to QM^m_\xi (\G)
\end{equation}
for each $\ga \in \wt{\G}$.

\begin{dfn} \label{D:9u}
The linear endomorphism \eqref{E:e5} given by \eqref{E:ox} is the {\em
Hecke operator\/} on $QM^m_\xi (\G)$ associated to $\ga \in \wt{\G}$.
\end{dfn}

\begin{rem}
Special types of Hecke operators on quasimodular forms as in Definition
\ref{D:9u} were considered by Movasati in \cite{Mo07} for $\G = SL(2,
\bZ)$.
\end{rem}

\begin{thm} \label{T:xk}
For each $\ga \in \wt{\G}$ the diagram
\begin{equation} \label{E:11}
\begin{CD}
QM^m_{\xi}(\G) @> \mQ^m_\xi >> QP^m_\xi (\G) @> \fS_m >>
M_{\xi -2m} (\G)\\
@V T^Q_{\xi} (\ga) VV @V T^P_\xi (\ga) VV @VV T_{\xi -2m} (\ga) V\\
QM^m_{\xi} (\G) @> \mQ^m_\xi >> QP^m_\xi (\G) @> \fS_m >> M_{\xi
-2m} (\G)
\end{CD}
\end{equation}
is commutative, where the maps $T^P_\xi (\ga)$, $T^Q_\xi (\ga)$ and
$T_{\xi -2m} (\ga)$ are as in \eqref{E:ju}, \eqref{E:ox} and \eqref{E:eb},
respectively.
\end{thm}

\begin{proof}
Let $\ga$ be an element of $\wt{\G}$ such that the corresponding double
coset has a decomposition as in \eqref{E:xm}. Given $f \in QM^m_\xi (\G)$,
using \eqref{E:ox}, we have
\[ (\mQ^m_\xi \circ T^Q_\xi (\ga)) f = \mQ^m_\xi \biggl( \fS_0 \biggl(
\sum^s_{i=1} (\mQ^m_\xi f) \md_\xi \ga_i \biggr) \biggr) = \sum^s_{i=1}
(\mQ^m_\xi f) \md_\xi \ga_i, \]
since $\mQ^m_\xi = \fS_0^{-1}$ by Proposition \ref{P:kw}.  On the other
hand, from \eqref{E:ju} we obtain
\[ (T^P_\xi (\ga) \circ Q^m_\xi) f = \sum^s_{i=1} (\mQ^m_\xi f) \md_\xi
\ga_i ;\]
hence it follows that
\[ \mQ^m_\xi \circ T^Q_\xi (\ga) = T^P_\xi (\ga) \circ Q^m_\xi .\]
We now consider a quasimodular polynomial $F(z,X) = \sum^m_{r=0} f_r (z)
X^r \in QP^m_\xi (\G)$.  Using \eqref{E:ue} and \eqref{E:ju}, we have
\begin{align*}
(\fS_m \circ T^P_\xi (\ga)) F &= \fS_m \biggl( \sum^s_{i=1} (F \md_\xi
\ga_i) \biggr)\\
&= \sum^s_{i=1} (\det \ga_i)^{\xi/2} \fJ (\ga_i, z)^{-\xi}\\
&\hspace{.7in} \times \fS_m \biggl( \sum^m_{r=0} f_r (\ga_i z)
(\det
\ga_i)^{-r} \fJ (\ga_i, z)^{2r} (X - \fK (\ga_i, z))^r \biggr)\\
&= \sum^s_{i=1} (\det \ga_i)^{\xi/2 -m} \fJ (\ga_i, z)^{2m -\xi} f_m
(\ga_i z)\\
&= \sum^s_{i=1} (\fS_m F) \mid_{\xi -2m} \ga_i = (T_{\xi -2m} (\ga) \circ
\fS_m) F.
\end{align*}
Thus we obtain
\[ \fS_m \circ T^P_\xi (\ga) = T_{\xi -2m} (\ga) \circ \fS_m ,\]
and therefore the proof of the theorem is complete.
\end{proof}

Given $\gl \in \bZ$, we now extend the action of $SL(2, \bR)$ in
\eqref{E:yt} to that of $GL^+ (2, \bR)$ by setting
\begin{equation} \label{E:sn}
(\Phi \mid^J_\gl \ga) (z,X) = \det (\ga)^{\gl/2} \fJ (\ga, z)^{-\gl}
e^{-\fK (\ga, z) X} \Phi (\ga z, (\det \ga) \fJ (\ga,z)^{-2} X)
\end{equation}
for $\Phi (z,X) \in \mF [[X]]$ and $\ga \in GL^+ (2,\mbb R)$. Then it can
be shown that
\begin{equation} \label{E:2hp}
(\Phi \mid^J_\gl \ga) \mid^J_\gl \ga' = \Phi \mid^J_\gl (\ga\ga')
\end{equation}
for $\ga, \ga' \in GL^+ (2,\mbb R)$.

If $\Phi (z,X) \in \mJ_\gl (\G)$ and if the double coset corresponding to
an element $\ga \in \wt{\G}$ has a decomposition as in \eqref{E:xm}, we set
\begin{equation} \label{E:jh}
(T^J_\gl (\ga) \Phi) (z,X) = \sum_{i=1}^s (\Phi \mid^J_\gl \ga_i) (z, X),
\end{equation}
for all $z \in \mH$.  Then the power series $(T^J_\gl (\ga) \Phi) (z,X)$ is
independent of the choice of the coset representatives $\ga_1, \ldots,
\ga_s$, and the map $\Phi \mapsto T^J_\gl (\ga) \Phi$ determines the
linear endomorphism
\[ T^J_\gl (\ga): \mJ_\gl (\G) \to \mJ_\gl(\G) .\]

\begin{thm} \label{T:ke}
If $\Pi^\gd_m$ is as in \eqref{E:mc}, for each $\ga \in \wt{\G}$ the
diagram
\[ \begin{CD}
\mJ_\gl (\G)_\gd @> \Pi^\gd_m >> QP^m_{\gl +2m +2\gd}(\G)\\
@V T^J_\gl (\ga) VV @VV T^P_{\gl +2m +2\gd} (\ga) V\\
\mJ_\gl (\G)_\gd @> \Pi^\gd_m >> QP^m_{\gl +2m +2\gd}(\G) ,
\end{CD}\]
is commutative, where the Hecke operators $T^P_{\gl +2m +2\gd} (\ga)$ and
$T^J_\gl (\ga)$ are as in \eqref{E:ju} and \eqref{E:jh}, respectively.
\end{thm}

\begin{proof}
Let $\ga$ be an element of $\wt{\G}$ such that the corresponding double
coset has a decomposition as in \eqref{E:xm}.  Using \eqref{E:ju} and
\eqref{E:jh} as well as the extension of \eqref{E:pv} to $GL^+ (2, \bR)$,
we obtain
\begin{align*}
(\Pi^\gd_m (T^J_\gl (\ga) \Phi)) (z,X) &= \sum^s_{i=1} \Pi^\gd_m (\Phi
\mid^J_\gl
\ga_i) (z,X)\\
&= \sum^s_{i=1} (\Pi^\gd_m \Phi) \md_{\gl +2m +2\gd} \ga_i) (z,X)\\
&= ((T^P_{\gl +2m +2\gd} (\ga) (\Pi^\gd_m \Phi)) (z,X)
\end{align*}
for all $\Phi (z,X) \in \mJ_\gl (\G)_\gd$, which shows the commutativity
of the given diagram.
\end{proof}

Given a modular form $f \in M_{2w +\gl} (\G)$, it is known that the formal
power series
\begin{equation} \label{E:pc}
\Phi_f (z, X) = \sum^\infty_{\ell =0} \frac {f^{(\ell +\gd
-w)}(z)} {(\ell -w +\gd)! (\ell +w +\gd +\gl -1)!} X^{\ell +\gd}
\end{equation}
with $p!=0$ for $p <0$ is a Jacobi-like form belonging to $\mJ_\gl
(\G)_\gd$(cf.\ \cite{EZ85}, \cite{Za94}).  Thus the map $f
\mapsto \Phi_f$ determines a lifting
\begin{equation} \label{E:mz}
\mL_{w,\gl}^X: M_{2w +\gl} (\G) \to \mJ_\gl (\G)_\gd
\end{equation}
of modular forms to Jacobi-like forms known as the Cohen-Kuznetsov lifting
(see \cite{CM97}).

Let $\pa = d/dz$ be the derivative operator on $\mF$.  If $h$ is a modular
form belonging to $M_k (\G)$, then it is known that $\pa^\ell h$ with
$\ell \geq 1$ is a quasimodular form belonging to $QM^\ell_{k +2\ell}
(\G)$ (see e.g.\ \cite{MR07}).

\begin{thm} \label{T:wh}
Let $f$ be a modular form belonging to $M_{2w + \gl} (\G)$ for some $w \in
\bZ$.

(i) The quasimodular form $\pa^{m-w +\gd} f \in QM^{m-w+\gd}_{\gl
+2m +2\gd} (\G)$ satisfies
\begin{equation} \label{E:sq}
\frac {\pa^{m-w +\gd} f} {(m-w +\gd)! (m + w +\gd +\gl -1)!} = (\fS_0
\circ \Pi^\gd_m \circ \mL^X_{w, \gl}) f .
\end{equation}

(ii) If $\ga \in \wt{\G}$, then we have
\begin{equation} \label{E:gc}
\pa^{m-w +\gd} (T_{2w +\gl} (\ga) f) = T^Q_{\gl +2m +2\gd} (\ga) \pa^{m-w
+\gd} f ,
\end{equation}
where $T_{2w +\gl} (\ga)$ and $T^Q_{\gl +2m +2\gd} (\ga)$ are the Hecke
operators in \eqref{E:fs} and \eqref{E:ox}, respectively.
\end{thm}

\begin{proof}
If $\mL_{w,\gl}^X$ is as in \eqref{E:mz}, we obtain the map
\[ \Pi^\gd_m \circ \mL_{w,\gl}^X: M_{2w +\gl} (\G) \to QP^m_{\gl + 2m +2\gd}
(\G) \]
given by%
\[ ((\Pi^\gd_m \circ \mL_{w,\gl}^X)f) (z, X) = \sum^m_{r =0}
\frac {f^{(m-r +\gd -w)} (z)} {(m-r -w +\gd)! (m-r +w +\gd +\gl
-1)!}\frac{X^r}{r!}
\]
for all $f \in M_{2w +\gl} (\G)$.  Thus we obtain \eqref{E:sq} by
considering the coefficient of $X^0$ in  this relation.  On the
other hand, given $\ga \in \wt{\G}$ and $f \in M_{2w +\gl} (\G)$,
it is known (cf.\ \cite{LM01}) that
\[ \mL_{w,\gl}^X (T_\gl (\ga) f) = T^J_\gl (\ga) (\mL_{w,\gl}^X f) .\]
Hence (ii) follows from this and Theorem \ref{T:ke}.
\end{proof}

Theorem \ref{T:wh} shows that the diagram
\[ \begin{CD}
M_{2w +\gl} (\G) @> \pa^{m-w +\gd} >> QM^m_{\gl +2m +2\gd} (\G)\\
@V T_{2w +\gl} (\ga) VV @VV T^Q_{\gl +2m +2\gd} (\ga) V\\
M_{2w +\gl} (\G) @> \pa^{m-w +\gd} >> QM^m_{\gl +2m +2\gd} (\G)
\end{CD}\]
is commutative for each $\ga \in \wt{\G}$.

\section{\bf{Liftings of modular forms to quasimodular forms}} \label{S:fq}

There is a natural surjective map from quasimodular polynomials to modular
forms sending a quasimodular polynomial to its leading coefficient.  In
this section we construct a lifting of a modular form to a quasimodular
form whose leading coefficient coincides with the given modular form.

If $\fS_m: \mF_m [X] \to \mF$ and $\wh{\fS}^\gd: \mF [[X]] \to \mF$ are as
in \eqref{E:v6}, from \eqref{E:hh} and Definition \ref{D:wk} we see that
\[
\wh{\fS}^\gd (\mJ_{\gl} (\G)_\gd) \subset M_{\gl +2\gd} (\G) ,\quad \fS_m
(QP^m_{\xi} (\G)) \subset  M_{\xi -2m} (\G) .
\]
Hence the rows in \eqref{E:dd} induce the short exact sequences of the
form
\[ 0 \to \mJ_{\gl} (\G)_{\gd +1} \xrightarrow{\wh{\iota}}  \mJ_{\gl}
(\G)_\gd \xrightarrow{\wh{\fS}^\gd} M_{\gl +2\gd} (\G) \to 0 ,\]
\[ 0 \to QP^{m-1}_\xi (\G) \xrightarrow{\iota} QP^m_{\xi} (\G)
\xrightarrow{\fS_m} M_{\xi -2m} (\G) \to 0 .\]
for integers $\xi$, $\gl$ and $\gd >0$. Furthermore, combining these
results with Proposition \ref{P:p4} and the isomorphism \eqref{E:37}, we
see that the diagram \eqref{E:dd} induces the commutative diagram
\begin{equation} \label{E:ee}
\begin{CD}
0 @>>> \mJ_{\gl} (\G)_{\gd +1} @> \wh{\iota} >>  \mJ_{\gl} (\G)_\gd
@> \wh{\fS}^\gd >> M_{\gl +2\gd} (\G) @>>> 0\\
@. @V \Pi^{\gd+1}_{m -1} VV @VV \Pi^\gd_m V @VV \mu_m V @.\\
0 @>>> QP^{m-1}_{\gl +2m + 2\gd} (\G) @> \iota >> QP^m_{\gl +2m + 2\gd}
(\G) @> \fS_m >> M_{\gl +2\gd} (\G) @>>> 0\\
@. @V \fS_0 VV @VV \fS_0 V @| @.\\
0 @>>> QM^{m-1}_{\gl +2m + 2\gd} (\G) @> >> QM^m_{\gl +2m + 2\gd} (\G) @>
>> M_{\gl +2\gd} (\G) @>>> 0
\end{CD}
\end{equation}
in which the three rows are short exact sequences.

\begin{prop}
Given integers $\gd$ and $\gl$ with $\gd >0$ and a modular form $h \in
M_{\gl+2\gd} (\G)$, the polynomial $(\Xi^{\gl +2\gd}_m h) (z, X) \in \mF_m [X]$
given by
\begin{equation} \label{E:ak}
(\Xi^{\gl +2\gd}_m h) (z,X) = \sum^m_{r=0} \frac {m! (\gl +2\gd -1)! h^{(m-r)}
(z)} {r! (m-r)! (m-r + 2\gd +\gl -1)!} X^r
\end{equation}
is a quasimodular polynomial belonging to $QP^m_{\gl +2m +2\gd} (\G)$.
Furthermore, we have
\begin{equation} \label{E:mu}
\fS_m (\Xi^{\gl +2\gd}_m h) = h .
\end{equation}
\end{prop}

\begin{proof}
Since $h \in M_{\gl+2\gd} (\G)$, by \eqref{E:pc} the formal power series
$\Phi_h (z, X)$ given by
\[ \Phi_h (z, X) = m! (\gl +2\gd -1)! \sum^\infty_{\ell =0} \frac
{h^{(\ell)}(z)} {\ell! (\ell + 2\gd +\gl -1)!} X^{\ell +\gd} \] is
a Jacobi-like form belonging to $\mJ_\gl (\G)_\gd$.  From this and
\eqref{E:mc} we obtain
\begin{align*}
(\Pi^{\gd}_m \Phi_h) (z, X) &= m! (\gl +2\gd -1)! \sum^m_{r=0} \frac
{ h^{(m-r)} (z)} {r! (m-r)! (m-r + 2\gd +\gl -1)!} X^r\\
&= (\Xi^{\gl +2\gd}_m h) (z,X) ,
\end{align*}
and it is a quasimodular polynomial belonging to $QP^m_{\gl +2m +2\gd}
(\G)$. On the other hand, we see easily that the coefficient of $X^m$ in
this polynomial is equal to $h$, which shows that $\fS_m (\Xi^{\gl
+2\gd}_m h) = h$.
\end{proof}

The formula \eqref{E:ak} determines the linear map
\[ \Xi^{\gl +2\gd}_m: M_{\gl+2\gd} (\G) \to QP^m_{\gl+2m +2\gd} (\G) ,\]
which may be regarded as a lifting from modular forms to quasimodular
forms.  From \eqref{E:mu} we see that the short exact sequence in the
second row of \eqref{E:ee} splits.  Furthermore, from \eqref{E:gc} and
\eqref{E:ak} we see that $\Xi^{\gl +2\gd}_m$ is Hecke equivariant, meaning
that the diagram
\begin{equation} \label{E:54}
\begin{CD}
M_{2\gd +\gl} (\G) @> \Xi^{\gl +2\gd}_m >> QP^m_{\gl +2m +2\gd} (\G)\\
@V T_{2\gd +\gl} (\ga) VV @VV T^P_{\gl +2m +2\gd} (\ga) V\\
M_{2\gd +\gl} (\G) @> \Xi^{\gl +2\gd}_m >> QP^m_{\gl +2m +2\gd} (\G)
\end{CD}
\end{equation}
is commutative for each $\ga \in \wt{\G}$.

We also consider the lifting
\[ \wh{\Xi}_\gd: M_{\gl+2\gd} (\G) \to \mJ_\gl (\G)_\gd \]
from modular forms to Jacobi-like forms defined by
\begin{align} \label{E:m2}
(\wh{\Xi}_\gd f) (z,X) &= (\gl +2\gd -1)! (\mL_{\gd,\gl}^X f) (z,X) \\
&= (\gl +2\gd -1)! \sum^\infty_{\ell =0} \frac {f^{(\ell)} (z)} {\ell!
(\ell + 2\gd +\gl -1)!} X^{\ell +\gd} \notag
\end{align}
for $f \in M_{\gl+2\gd} (\G)$, where $\mL_{\gd,\gl}^X$ is the
Cohen-Kuznetsov lifting in \eqref{E:mz}.  Then we have
\begin{equation} \label{E:r6}
\wh{\fS}^\gd \circ \wh{\Xi}_\gd (f) = f.
\end{equation}
It can be shown that the map $\wh{\Xi}_\gd$ is Hecke
equivariant (cf.\ \cite{LM01}), so that the diagram
\[\begin{CD}
M_{2\gd +\gl} (\G) @> \wh{\Xi}_\gd >> \mJ_\gl (\G)_\gd\\
@V T_{2w +\gl} (\ga) VV @VV T^J_\gl (\ga) V\\
M_{2\gd +\gl} (\G) @> \wh{\Xi}_\gd >> \mJ_\gl (\G)_\gd
\end{CD}\]
is commutative for each $\ga \in \wt{\G}$.  Furthermore, we also obtain the
diagram
\begin{equation} \label{E:j4}
\begin{CD}
M_{\gl +2\gd} (\G) @> \wh{\Xi}_\gd >> \mJ_{\gl} (\G)_\gd\\
@VV \mu_m V @VV \Pi^\gd_m V\\
M_{\gl +2\gd} (\G) @> \Xi^{\gl +2\gd}_m >> QP^m_{\gl +2m + 2\gd} (\G)
\end{CD}
\end{equation}
that is commutative.  We note that the linear map
\[ \wh{\Xi}_\gd \circ \fS_m: QP^m_{\gl +2m + 2\gd} (\G) \to
\mJ_{\gl} (\G)_\gd \]
satisfies
\[ \wh{\fS}^\gd \circ \wh{\Xi}_\gd \circ \fS_m = \fS_m ,\]
which follows from \eqref{E:r6} and the relation $\fS_m (QP^m_{\gl +2m +
2\gd} (\G)) \subset M_{\gl+2\gd} (\G)$.

\section{\bf{Liftings of quasimodular forms to Jacobi-like forms}} \label{S:jq}

In Section \ref{S:fq} we considered liftings of modular forms to
quasimodular forms as well as Cohen-Kuznetsov liftings of modular forms to
Jacobi-like forms that are Hecke equivariant.  On the other hand, in
Section \ref{S:qf} we studied surjective maps from Jacobi-like forms to
quasimodular polynomials.  In this section we construct liftings of
quasimodular to Jacobi-like forms corresponding to this surjection that are
Hecke equivariant.

It is important to note that the map $\Pi^\gd_m: \mJ_\gl (\G)_\gd \to
QP^m_{\gl +2m +2\gd} (\G)$ in \eqref{E:7g} depends on $\gd$.  For example,
if $\ve$ is an integer with $0 \leq \ve \leq \gd$ a Jacobi like form
\[ \Phi (z,X) = \sum^\infty_{k =0} \phi_k (z) X^{k+\gd} \in \mJ_\gl
(\G)_\gd \]
can be regarded as an element of $\mJ_\gl (\G)_\ve$ by writing it as
\[ \Phi (z,X) = \sum^\infty_{k =0} \phi_{k -\gd +\ve} (z) X^{k+\ve} \]
with $\phi_j =0$ for $j < 0$.  Thus, from \eqref{E:mc} we see that
\[ (\Pi^\gd_m \Phi) (z, X) = \sum^m_{r =0} \frac{1}{r!} \phi_{m-r} (z) X^r
,\]
\begin{align*}
(\Pi^\ve_m \Phi) (z, X) &= \sum^m_{r =0} \frac{1}{r!} \phi_{m -\gd
  +\ve -r} (z) X^r\\
&= \sum^{m-\gd +\ve}_{r =0} \frac{1}{r!} \phi_{m -\gd
  +\ve -r} (z) X^r\\
&= (\Pi^\ve_{m-\gd +\ve} \Phi) (z, X)
\end{align*}
Although $(\Pi^\ve_m \Phi) (z, X)$ above belongs to $QP^{m-\gd +\ve}_{\gl
+2m +2\gd} (\G) \subset QP^{m}_{\gl +2m +2\gd} (\G)$ and involves only the
first $m-\gd +\ve +1$ coefficients of $\Phi (z, X)$, it is different from
$(\Pi^\gd_{m-\gd +\ve} \Phi) (z, X)$ which is given by
\[ (\Pi^\gd_{m-\gd +\ve} \Phi) (z, X) = \sum^{m-\gd +\ve}_{r =0}
\frac{1}{r!} \phi_{m -r} (z) X^r .\]

In order to construct a lifting from quasimodular polynomials to
Jacobi-like forms, as in the next lemma, we first determine a quasimodular
polynomial of degree $\leq m-1$ from one of degree $\leq m$ by using the
commutative diagram \eqref{E:j4}.  Combining this with the fact that
quasimodular forms of degree 0 are modular forms, the desired lifting can
be obtained by induction and Cohen-Kuznetsov liftings of modular forms (see
Theorem \ref{T:mn} below).

\begin{lem}
Let $F(z,X)$ be a quasimodular polynomial belonging to $QP^m_{\xi} (\G)$,
and let $m$ and $\xi$ be integers with $m \geq 1$.  Then we have
\begin{equation} \label{E:9k}
F (z,X) - m! ((\Pi^\gd_m \circ \wh{\Xi}_\gd \circ \fS_m) F) (z, X) \in
QP^{m-1}_\xi (\G),
\end{equation}
where $\fS_m$, $\wh{\Xi}_\gd$ and $\Pi^\gd_m$ are as in \eqref{E:hh},
\eqref{E:7g} and \eqref{E:m2}, respectively.
\end{lem}

\begin{proof}
The proof follow from \eqref{E:mu} and the commutative diagram
\eqref{E:j4}.  Indeed, given $F (z, X) \in QP^m_{\xi} (\G)$, we have $\fS_m
F \in M_{\xi -2m} (\G)$, hence we see that
\[ ((\wh{\Xi}_\gd \circ \fS_m) F) (z, X) \in \mJ_{\xi -2m -2\gd} (\G)_\gd
.\]
Then the initial term of this Jacobi-like form is equal to $(\fS_m F) (z)$,
which implies that the coefficient of $X^m$ in the quasimodular polynomial
\[ ((\Pi^\gd_m \circ \wh{\Xi}_\gd \circ \fS_m) F) (z, X) \in
QP^m_{\xi} (\G) \]
is $\frac 1{m!} (\fS_m F) (z)$.  On the other hand, the corresponding
coefficient in $F (z,X)$ is equal to $(\fS_m F) (z)$; hence the lemma
follows.
\end{proof}

\begin{thm} \label{T:mn}
Given a positive integer $m$ and a quasimodular polynomial $F(z,X) \in
QP^m_{\xi} (\G)$ with $\xi > 2m \geq 2$, there exists a Jacobi-like form
$\Phi (z,X) \in \mJ_{\xi -2m -2\gd} (\G)_\gd$ such that%
\[ \Pi^\gd_m \Phi = F \]
for each $\gd \geq 0$.
\end{thm}

\begin{proof}
We first consider the case where $m=1$.  Let
\[ F (z,X) = f_0 (z) + f_1(z) X \in QP^1_\xi (\G) ,\]
so that $f_1 = \fS_1 F \in M_{\xi -2} (\G)$.  Using \eqref{E:9k}, we see
that there is an element $h \in QP^0_\xi (\G) = M_\xi (\G)$ such that
\[ F (z, X) - \Pi^\gd_1 (\wh{\Xi}_\gd f_1) (z, X) = h (z) \]
for all $z \in \mH$.  However, from \eqref{E:r6} we see that
\[ h = \Pi_0^{\gd+1} (\wh{\Xi}_{\gd +1} h) \]
with $\wh{\Xi}_{\gd +1} h \in \mJ_{\xi -2 -2\gd} (\G)_{\gd +1}$.  If we
write
\[ (\wh{\Xi}_\gd f_1) (z,X) = \sum^\infty_{k=0} \phi_k (z) X^{k +\gd},
\quad (\wh{\Xi}_{\gd +1} h) (z,X) = \sum^\infty_{k=0} \psi_k (z) X^{k +\gd
  +1} \]
with $\phi_k$ and $\psi_k$ being holomorphic functions for each $k \geq 0$,
then we obtain
\[ F (z,X) = (\Pi^\gd_1 (\wh{\Xi}_\gd f_1) + \Pi_0^{\gd+1} (\wh{\Xi}_{\gd +1} h))
(z,X) = (\phi_1 + \psi_0) (z) + (\phi_0 + \psi_{-1}) (z) X \] with
$\psi_{-1} =0$.  Thus we see that
\[ F = \Pi^\gd_1 (\wh{\Xi}_\gd f_1 + \wh{\Xi}_{\gd +1} h) ,\]
where
\[ (\wh{\Xi}_\gd f_1 + \wh{\Xi}_{\gd +1} h) (z,X) = \sum^\infty_{k=0}
(\phi_k + \psi_{k-1}) (z) X^{k +\gd} \in \mJ_{\xi -2 -2\gd}
(\G)_\gd \] with $\psi_{-1} =0$; hence the $m=1$ case follows.  We
now consider an integer $m>1$ and assume that the statement in the
theorem holds for all positive integers less than $m$.  If $F(z,X)
\in QP^m_{\xi} (\G)$, then by \eqref{E:9k} there is an element $G
(z,X) \in QP^{m-1}_\xi (\G)$ satisfying
\[ F - (m!) \Pi^\gd_m (\wh{\Xi}_\gd (\fS_m F)) = G .\]
Thus by the induction hypothesis there is an element $\Psi (z,X) \in
\mJ_{\xi -2m -2\gd} (\G)_{\gd +1}$ such that%
\[ \Pi^{\gd +1}_{m -1} \Psi = G .\]
Hence we obtain
\[ F = (m!) \Pi^\gd_m (\wh{\Xi}_\gd (\fS_m F)) + \Pi^{\gd
  +1}_{m -1} \Psi .\]
Thus, assuming that
\[ (m!) (\wh{\Xi}_\gd (\fS_m F)) (z,X) = \sum^\infty_{k=0} \ga_k (z) X^{k
+\gd}, \quad \Psi (z,X) = \sum^\infty_{k=0} \gb_k (z) X^{k +\gd +1} \]
with $\ga_k$ and $\gb_k$ holomorphic for each $k \geq 0$, we have
\begin{align} \label{E:x9}
F(z,X) &= \sum^m_{r=0} \frac 1{r!} \ga_{m-r} (z) X^r + \sum^{m -1}_{r=0}
\frac 1{r!} \gb_{m -1-r} (z) X^r\\
&= \sum^m_{r=0} \frac 1{r!} (\ga_{m-r} + \gb_{m -1-r}) (z) X^r \notag
\end{align}
with $\gb_{-1} =0$.  If we set
\begin{equation} \label{E:v3}
\Phi (z,X) = \bigl( (m!)\wh{\Xi}_\gd (\fS_m F) + \Psi \bigr) (z,X) =
\sum^\infty_{k=0} (\ga_k +\gb_{k-1}) (z) X^{k+\gd}
\end{equation}
with $\gb_{-1} =0$, then $\Phi (z,X)$ is a Jacobi-like form belonging to
$\mJ_{\xi -2m -2\gd} (\G)_\gd$.  Using this and \eqref{E:x9}, we obtain
\[ F = \Pi^\gd_m \Phi .\]
Thus the statement is true for $m$, and the proof of the theorem is
complete.
\end{proof}

Given a quasimodular polynomial $F_m(z,X) \in QP^m_{\xi} (\G)$ with $\xi
> 2m \geq 0$ and $\gd \geq 0$, the proof of Theorem \ref{T:mn}
provides us with a canonical way of obtaining a Jacobi-like form
\[ \Phi_m (z,X) = \sum^\infty_{k=0} \phi_{m, k} (z) X^{k +\gd} \in
\mJ_{\xi -2m -2\gd} (\G)_\gd \]
satisfying $\Pi^\gd_m \Phi_m = F_m$.  Indeed, the coefficients of $\Phi_m
(z,X)$ can be determined recursively as follows.  Using \eqref{E:m2}, we
have
\[ (m!) (\wh{\Xi}_\gd (\fS_m F_m)) (z,X) = m! (\xi -2m -1)! \sum^\infty_{k=0}
\frac {(\fS_m F_m)^{(k)} (z)} {k! (k + \xi -2m -1)!} X^{k +\gd} .\]
From this and \eqref{E:v3} we obtain
\[ \Phi_m (z,X) = \sum^\infty_{k=0} \biggl( \frac {m! (\xi -2m -1)!}
{k! (k + \xi -2m -1)!}  (\fS_m F_m)^{(k)} + \phi_{m-1, k-1}\biggr) (z)
X^{k +\gd} .\]
Thus we see that
\[ \phi_{m,k} = \frac {m! (\xi -2m -1)!} {k! (k + \xi -2m -1)!}
(\fS_m F_m)^{(k)}+ \phi_{m-1, k-1} \]
for all $k, m \geq 1$ with
\[ \phi_{0,k} = \frac {F_0^{(k)}} {k! (k + \xi -1)!} ;\]
here we note that $F_0 (z,X)$ does not contain $X$ and is a function of
$z$ only, so that
\[ F^{(k)}_0 (z,X) = \frac {\pa^k} {\pa z^k} F_0 (z,X) \]
for $k \geq 1$.

From the above construction we see that the map $F_m \mapsto \Phi_m$
determines a linear map
\begin{equation} \label{E:3p}
\mL^\gd_{m, \xi}: QP^m_\xi (\G) \to \mJ_{\xi -2m -2\gd} (\G)_\gd ,
\end{equation}
which may be considered as a canonical lifting from quasimodular forms to
Jacobi-like forms such that
\begin{equation} \label{E:nn}
(\Pi^\gd_m \circ \mL^\gd_{m,\xi}) F = F
\end{equation}
for all $F(z,X) \in QP^m_\xi (\G)$. In the case where $m=0,$ we
define $ \mL^\gd_{0, \xi}$ to be the map
\[ \wh{\Xi}_\gd: M_{\xi} \to \mJ_{\xi-2\gd}(\G)_{\gd} .\]
From the Hecke equivariance of
the maps $\Pi^\gd_m$, $\wh{\Xi}_\gd$ and $\fS_m$ we obtain the
same property for $\mL^\gd_{m, \xi}$, so that the diagram
\[\begin{CD}
QP^m_\xi (\G) @> \mL^\gd_{m, \xi} >> \mJ_{\xi -2m -2\gd} (\G)_\gd\\
@V T^Q_\xi (\ga) VV @VV T^J_{\xi -2m -2\gd} (\ga) V\\
QP^m_\xi (\G) @> \mL^\gd_{m, \xi} >> \mJ_{\xi -2m -2\gd} (\G)_\gd
\end{CD}\]
is commutative for $\ga \in \wt{\G}$.  We can also consider the lifting
\[ \wt{\mL}^\gd_{m, \xi}: QM^m_\xi (\G) \to \mJ_{\xi -2m -2\gd} (\G)_\gd \]
of quasimodular forms given by
\[ \wt{\mL}^\gd_{m, \xi} = \mL^\gd_{m, \xi} \circ \mQ^\gd_\xi ,\]
where $\mQ^\gd_\xi$ is as in \eqref{E:tp}.

Given integers $m$, $\gd$ and $\xi$ with $\gd, m \geq 0$, there is a short
exact sequence of the form
\begin{equation} \label{E:wb}
0 \to \mJ_{\xi -2m -2\gd} (\G)_{\gd +m} \xrightarrow{\iota}  \mJ_{\xi -2m
-2\gd} (\G)_\gd \xrightarrow{\Pi^\gd_m} QP^m_{\xi} (\G) \to 0 ,
\end{equation}
where $\iota$ is the inclusion map.  From the relation \eqref{E:nn} it
follows that the short exact sequence in \eqref{E:wb} splits.  Similarly,
there is also a split short exact sequence of the form
\[ 0 \to \mJ_{\xi -2m -2\gd} (\G)_{\gd +m} \xrightarrow{\iota}  \mJ_{\xi -2m
-2\gd} (\G)_\gd \xrightarrow{\wt{\Pi}^\gd_m} QM^m_{\xi} (\G) \to 0 ,\]
where $\wt{\Pi}^\gd_m = \fS_0 \circ \Pi^\gd_m$.

\begin{rem}
Since $\fS_mF \in M_{\xi -2m} (\G)$ for $F (z,X) \in QP^m_\xi (\G)$, we
see that $\fS_m F =0$ if $\xi < 2m$.  Thus the only case that was not
covered in Theorem \ref{T:mn} is when $\xi = 2m$.  As an example of a
quasimodular form for $\xi = 2m$, we can consider the Eisenstein series
$E_2$ given by
\begin{equation} \label{E:e2}
E_2 (z) = 1- 24 \sum^\infty_{n=1} \gs_1 (n) e^{2\pi i nz} ,
\end{equation}
where $\gs_1 (n) = \sum_{d \mid n} d$.  Then it is known that $E_2$
satisfies the relation
\begin{equation} \label{E:e3}
(E_2 \mid_2 \g) (z) = E_2 (z) + \frac 6{i\pi} \fK(\g,z)
\end{equation}
for all $z \in \mH$ and $\g \in SL(2, \mb Z)$.  Hence $E_2$ is a
quasimodular form belonging to $QM^1_2 (\G)$ with $\G = SL(2, \mb Z)$, and
the associated quasimodular polynomial is given by
\[ (\mQ^1_2 E_2) (z, X) = E_2 (z) + \frac 6{i\pi} X \in QP^1_2
(\G) .\]
Using $\g = \sm 0&-1\\ 1&0 \esm$, the relation \eqref{E:e3} can be written
in the form
\[ E_2 (-1/z) = z^2 E_2(z) +\frac {12z}{2\pi i} .\]
Using this and induction, it can be shown that
\begin{align*}
\frac{E_2^{(n)} (-1/z)} {n! (n+1)!} &= \sum_{\ell=0}^n  \frac{
E_2^{(\ell)}(z) }{(2\pi i)^{n-\ell}(n-\ell)!\ell!(\ell+1)! }z^{2+\ell+n} +
\frac{12 z^{n+1}}
{(2\pi i)^{n+1}(n+1)! }\\
&=\sum_{\ell=-1}^n \frac{ E_2^{(\ell)}(z)} {(2\pi i)^{(n-\ell)}(n-\ell)!
\ell!(\ell+1)!}z^{2+\ell+n},
\end{align*}
where $(-1)!= 1/12$ and $E_2^{(-1)}(z)=1$.  We now set
\[ \Phi_{E_2} (z, X) =\sum_{\ell =-1}^{\infty} \frac {E_{2}^{(\ell)}(z)}
{\ell!(\ell+1)!}(2\pi i X)^{\ell}. \]
Then it can be shown that
\[ \Phi_{E_2} (-1/z, z^{-2} X) = z^2 e^{X/z}
\Phi_{E_2} (z,X) ,\quad \Phi_{E_2} (z +1, X) = \Phi_{E_2} (z,X) ;\]
hence we have
\[ \Phi_{E_2} (\g z, \fJ (\g, z)^{-2} X) = \fJ (\g, z)^2
e^{- \fK (\g,z) X} \Phi_{E_2} (z,X) \]
for $\g = \sm 0&-1\\ 1&0 \esm, \sm 1&1\\ 0&1 \esm$.  Since $SL(2, \bZ)$ is
generated by $\sm 0&-1\\ 1&0 \esm$ and $\sm 1&1\\ 0&1 \esm$, we see that
$\Phi_{E_2} (z, X)$ is a Jacobi-like form belonging to $\mJ_2 (\G)$.  On
the other hand, we have
\[ \Pi_1^{0} (\Phi_{E_2}) = \mQ^1_2 E_2 ,\]
and therefore the Jacobi-like form $\Phi_{E_2} (z,X)$ is a lifting of the
quasimodular polynomial $\mQ^1_2 E_2 (z,X)$.
\end{rem}

\begin{ex}
Let $f \in M_w (\G)$ be a modular form of weight $w >0$.  Then its $m$-th
derivative with $m >0$ is a quasimodular form belonging to $QM^m_{w+2m}
(\G)$, and it can be shown that
\[ (f^{(m)} \mid_{w +2m} \g) (z) = \sum^m_{k=0} k! \bi mk \bi {w +m-1}k
f^{(m-k)} (z) \fK (\g, z)^k \]
for all $z \in \mH$ and $\g \in \G$ (see e.g.\ \cite{MR07}).  Thus, if we
denote the corresponding quasimodular polynomial by $F(z,X) \in QP^m_{w +2m}
(\G)$, we have
\[ F (z,X) = (\mQ^m_{w +2m} f^{(m)}) (z,X) = \sum^m_{k=0} k! \bi mk
\bi {w +m-1}k f^{(m-k)} (z) X^k ,\]
so that
\[ \fS_m F = m! \bi {w+m-1}m f \in M_w (\G) ,\]
\begin{equation} \label{E:61}
\wh{\Xi}_\gd (\fS_m F) (z,X) = (w+m-1)! \sum^\infty_{\ell =0} \frac
{f^{(\ell)} (z)} {\ell! (\ell +w -1)!} X^{\ell +\gd} .
\end{equation}
Hence we see that
\begin{align*}
\Pi_m^{\gd} (\wh{\Xi}_\gd (\fS_m F)) (z,X) &= (w+m-1)! \sum^m_{r
=0} \frac
{f^{(m-r)} (z)} {r! (m-r)! (m-r +w -1)!} X^r\\
&= \frac{1}{m!} F(z,X),
\end{align*}
and therefore the Jacobi-like form in \eqref{E:61} is a lifting of the
quasimodular polynomial $(\mQ^m_{w +2m} f^{(m)}) (z,X)$.
\end{ex}

\begin{ex}
Let $E_2$ be the Eisenstein series in \eqref{E:e2}.  We consider a modular
form $f \in M_w (\G)$ with $w >0$, and set
\[ F(z,X) = f(z) (\mQ^1_2 E_2) (z) = f(z) E_2 (z) + \frac 6{i\pi} f(z) X
,\]
which is a quasimodular polynomial belonging to $QP^1_{w+2} (\G)$.  Then
we have
\[ \wh{\Xi}_{\gd} (\fS_1 F) (z,X) = \frac 6{i\pi} (w-1)! \sum^\infty_{\ell =0} \frac
{f^{(\ell)} (z)} {\ell! (\ell +w -1)!} X^{\ell +\gd} ,\]
so that
\[ \Pi^\gd_1 (\wh{\Xi}_{\gd} (\fS_1 F)) (z,X) = \frac 6{i\pi} (w-1)! \biggl(
\frac {f'(z)} {w!} + \frac {f(z)} {(w-1)!} X \biggr) = \frac
6{i\pi w} f' (z) + \frac 6{i\pi} f(z) X . \] Thus we obtain
\[ F(z,X) - \Pi^\gd_1 (\wh{\Xi}_{\gd} (\fS_1 F)) (z,X) = f(z) E_2 (z)
- \frac 6{i\pi w} f' (z) \in QP^0_{w+2} (\G) = M_{w+2} (\G) ,\]
\[ \wh{\Xi}_{\gd +1} (F - \Pi^\gd_1 (\wh{\Xi}_{\gd} (\fS_1 F))) (z,X) = (w+1)!
\sum^\infty_{\ell =0} \frac {1} {\ell! (\ell +w +1)!} \Bigl( f E_2
- \frac 6{i\pi w} f' \Bigr)^{(\ell)} X^{\ell +\gd} .\] Hence we
see that $F(z,X) = \Pi^\gd_1 \Phi (z,X)$ with
\[ \Phi (z,X) = \sum^\infty_{k=0} \phi_k (z) X^{k+\gd+1} ,\]
where $\phi_0 = (6/i\pi) f$ and
\begin{align*}
\phi_k &= \frac {6 (w-1)!} {i\pi k! (k+w-1)!} f^{(k)} + \frac {(w+1)!}
{(k-1)! (k+w)!} \Bigl( f E_2 - \frac 6{i\pi w} f' \Bigr)^{(k-1)}\\
&= \frac {(w-1)!} {k! (k+w)!} \biggl( \frac 6{i\pi} (k+w) f^{(k)} + k! w
(w+1) \Bigl( f E_2 - \frac 6{i\pi w} f' \Bigr)^{(k-1)} \biggr)
\end{align*}
for all $k \geq 1$.
\end{ex}

The canonical lifting map $\mL^\gd_{m,\xi}$ in \eqref{E:3p} can also be
used to determine linear maps of quasimodular forms.  Indeed, if $n$ is
another nonnegative integer, we set
\[ \fT^\gd_\xi (m,n) = \Pi^\gd_n \circ \mL^\gd_{m,\xi} .\]
Then since
\[ \Pi^\gd_n (\mJ_{\xi -2m -2\gd} )(\G)_\gd \to QP^n_{\xi +2n -2m} (\G) ,\]
we obtain a linear map
\[ \fT^\gd_\xi (m,n): QP^m_\xi (\G) \to QP^n_{\xi +2n -2m} (\G) \]
of quasimodular forms such that $\fT^\gd_\xi (m,m)$ is the identity map on
$QP^m_\xi (\G)$.

\section{\bf{Pseudodifferential operators}} \label{S:jj}

Jacobi-like forms for a discrete subgroup $\G \subset SL(2, \bR)$ are known
to be in one-to-one correspondence with $\G$-automorphic pseudodifferential
operators (see \cite{CM97}).  The noncommutative multiplication operation
on the space of pseudodifferential operators given by the Leibniz rule
determines a natural Lie algebra structure on the same space.  In this
section we use the above correspondence and the liftings discussed in
Section \ref{S:jq} to construct Lie brackets on the space of quasimodular
polynomials.

A pseudodifferential operator over $\mF$ is a formal Laurent series in the
formal inverse $\pa^{-1}$ of $\pa = d/dz$ with coefficients in $\mF$ of
the form
\[ \Psi (z) = \sum^u_{k= -\infty} h_k (z) \pa^k \]
with $u \in \mbb Z$ and $h_k \in \mF$ for each $k \leq u$.  We denote by
$\sdf$ the space of all pseudodifferential operators over $\mF$. Then the
group $SL(2, \bR)$ acts on $\sdf$ on the right by
\[
(\Psi\circ \g) (z)  = \Psi (\g z) = \sum^u_{k= -\infty} h_k (\g z) (\fJ
(\g,z)^2 \pa)^k
\]
for all $\g \in SL(2, \bC)$, where $\fJ (\g,z)$ is as in \eqref{E:kz}.

\begin{dfn}
Given a discrete subgroup $\G$ of $SL(2, \bR)$, an element $\Psi (z) \in
\sdf$ is an {\em automorphic pseudodifferential operator for $\G$\/} if it
satisfies
\[ \Psi \circ \g = \Psi \]
for all $\g \in \G$.  We denote by $\sdf^\G$ the space of all automorphic
pseudodifferential operators for $\G$.
\end{dfn}

If $\ga$ is an integer, we denote by $\sdf_\ga$ the subspace of $\sdf$
consisting of the pseudodifferential operators of the form
\[ \sum^\infty_{k=0} \psi_k (z) \pa^{\ga -k} \]
with $\psi_k \in \mF$ for all $k \geq 0$.  We also set
\[ \sdf_\ga^\G = \sdf_\ga \cap \sdf^\G .\]

We now consider an isomorphism between the space of formal power series
and that of pseudodifferential operators following the idea of Cohen,
Manin, and Zagier in \cite{CM97}.  We shall see below that this isomorphism
is equivariant with respect to certain actions of $SL(2, \bR)$.  Given a
formal power series
\[
F (z, X) = \sum^\infty_{k=0} f_k (z) X^{k+\gd} \in \mF[[X]]_\gd
\]
and a pseudodifferential operator
\begin{equation} \label{E:7v}
\Psi (z) = \sum^\infty_{k=0} \psi_{k} (z) \pa^{-k-\ve} \in \sdf_{-\ve}
\end{equation}
with $\gd, \ve > 0$, we set
\begin{equation} \label{E:nh}
(\mI_\xi^\pa F) (z) = \sum^\infty_{k=0} C_{k+\gd+\xi} f_k (z) \pa^{-k
  -\gd-\xi},
\end{equation}
\begin{equation} \label{E:nm}
(\mI_\xi^X \Psi) (z, X) = \sum^\infty_{k=0} C_{k+\ve}^{-1} \psi_k (z) X^{k
+\ve -\xi}
\end{equation}
for each nonnegative integer $\xi$, where $C_\eta$ with $\eta > 0$ denotes
the integer
\begin{equation} \label{E:cf}
C_\eta = (-1)^\eta \eta! (\eta-1)! .
\end{equation}
Then it can be easily seen that
\begin{equation} \label{E:th}
(\mI^X_\xi \circ \mI^\pa_\xi) F = F, \quad (\mI^\pa_\xi \circ \mI^X_\xi)
\Psi = \Psi,
\end{equation}
and therefore we obtain the maps
\[
\mI^\pa_\xi: \mF[[X]]_\gd \to \sdf_{-\gd -\xi}, \quad \mI^X_\xi:
\sdf_{-\ve} \to \mF[[X]]_{\ve -\xi}
\]
that are complex linear isomorphisms.  The following proposition shows that
these isomorphisms are $SL(2, \bR)$-equivariant.

\begin{prop} \label{P:hw}
Let $F (z, X) \in \mF[[X]]_\gd$ and $\Psi (z) \in \sdf_{-\ve}$, and let
$\xi$ be a nonnegative integer.  Then we have
\[
(\mI_\xi^\pa F) \circ \g = \mI_\xi^\pa (F \mid^J_{2\xi} \g), \quad
(\mI_\xi^X \Psi) \mid^J_{2\xi} \g = \mI_\xi^X (\Psi \circ \g)
\]
for all $\g \in SL(2,\bR)$.
\end{prop}

\begin{proof}
See \cite[Proposition 1.1]{L07a}.
\end{proof}

\begin{cor} \label{C:3k}
Let $F (z, X)$, $\Psi (z)$ and $\xi$ be as in Proposition \ref{P:hw}.

(i) The formal power series $F (z, X)$ is a Jacobi-like form belonging to
$\mJ_{2 \xi} (\G)_\gd$ if and only if $\mI_\xi^\pa F \in \sdf_{-\gd
-\xi}^\G$.

(ii) The pseudodifferential operator $\Psi (z)$ belongs to
$\sdf_{-\ve}^\G$ if and only if $\mI_\xi^X \Psi \in \mJ_{2 \xi} (\G)_{\ve
-\xi}$.
\end{cor}

\begin{proof}
This follows immediately from Proposition \ref{P:hw}.
\end{proof}

If $\Psi (z) \in \sdf_{-\ve}$ is of the form \eqref{E:7v}, we set
\begin{equation} \label{E:z8}
(\PPi^{-\ve}_m \Psi) (z,X) = \sum^m_{r=0} \frac {\psi_{m-r} (z)} {r!
C_{m-r+\ve}} X^r ,
\end{equation}
where $C_{m-r+\ve}$ is as in \eqref{E:cf}.

\begin{lem} \label{L:5x}
If $\Psi (z) \in \sdf_{-\ve}^\G$ and $F(z,X) \in \mJ_{2\xi} (\G)_{\ve
-\xi}$, then we have
\begin{equation} \label{E:17}
\PPi^{-\ve}_m \Psi = \Pi^{\ve -\xi}_m (\mI^X_\xi \Psi), \quad \Pi^{\ve
-\xi}_m F = \PPi^{-\ve}_m (\mI^\pa_\xi F)
\end{equation}
for all $\xi \in \bZ$, where $\Pi^{\ve -\xi}_m$ is as in \eqref{E:7g}.  In
particular, $(\PPi^{-\ve}_m \Psi) (z,X)$ is a quasimodular polynomial
belonging to $QP^m_{2m +2\ve}$.
\end{lem}

\begin{proof}
Let $\Psi (z) \in \sdf_{-\ve}$ be as in \eqref{E:7v}, so that $(\mI^X_\xi
\Psi) (z,X)$ is given by \eqref{E:nm}.  Using \eqref{E:mc}, we obtain
\[ (\Pi^{\ve -\xi}_m (\mI^X_\xi \Psi)) (z, X) = \sum^m_{r=0}
\frac 1{r!} C_{m-r+\ve}^{-1} \psi_{m-r} (z) X^r = (\PPi^{-\ve}_m \Psi)
(z,X) .\]
Since $(\mI^X_\xi \Psi) (z,X) \in \mJ_{2\xi} (\G)_{\ve -\xi}$, we see that
$(\Pi^{\ve -\xi}_m (\mI^X_\xi \Psi)) (z, X)$ belongs to $QP^m_{2m +2\ve}$
by Proposition \ref{P:p4} and that the first relation in \eqref{E:17}
holds.  The second relation follows from this and Corollary \ref{C:3k}.
\end{proof}

By Lemma \ref{L:5x} there is a linear map
\[ \PPi^{-\ve}_m: \sdf_{-\ve}^\G \to QP^m_{2m +2\ve} \]
such that the diagram
\[
\begin{CD}
\sdf_{-\ve}^\G @> \mI^X_\xi >> \mJ_{2\xi} (\G)_{\ve -\xi}\\
@V \PPi^{-\ve}_m VV @VV \Pi^{\ve -\xi}_m V\\
QP^m_{2m +2\ve} (\G) @= QP^m_{2m +2\ve} (\G)
\end{CD}
\]
is commutative.  We also see that there is a short exact sequence of the
form
\begin{equation} \label{E:0b}
0 \to \sdf_{m-\ve}^\G \to \sdf_{-\ve}^\G \xrightarrow{\PPi^{-\ve}_m}
QP^m_{2m +2\ve} (\G) \to 0 .
\end{equation}
Similarly, there is also a split short exact sequence of the form
\begin{equation} \label{E:6b}
0 \to \sdf_{m-\ve}^\G \to \sdf_{-\ve}^\G
\xrightarrow{{^\pa}\wt{\Pi}^{-\ve}_m} QM^m_{2m +2\ve} (\G) \to 0 ,
\end{equation}
where ${^\pa}\wt{\Pi}^\gd_m = \fS_0 \circ \PPi^\gd_m$. If
\[ \mL^\gd_{m, 2m +2\ve}: QP^m_{2m +2\ve} (\G) \to \mJ_{2\ve -2\gd} (\G)_\gd \]
is as in \eqref{E:3p}, we set
\begin{equation} \label{E:3t}
{^\pa}\mL^\gd_{m, 2m +2\ve} = \mI^\pa_{\ve -\gd} \circ \mL^\gd_{m, 2m
  +2\ve}: QP^m_{2m +2\ve} (\G) \to \sdf_{-\ve}^\G
\end{equation}
Using \eqref{E:nn}, \eqref{E:17} and \eqref{E:3t}, we have
\begin{align*}
(\PPi^{-\ve}_m \circ {^\pa}\mL^\gd_{m, 2m +2\ve}) F &= (\Pi^\gd_m \circ
\mI^X_{\ve -\gd}) \circ (\mI^\pa_{ve-\gd} \circ \mL^\gd_{m, 2m +2\ve}) F\\
&= (\Pi^\gd_m \circ \mL^\gd_{m, 2m +2\ve}) F = F
\end{align*}
for all $F(z,X) \in QP^m_{2m +2\ve} (\G)$.  Thus the map ${^\pa}\mL^\gd_{m,
2m +2\ve}$ is a lifting of quasimodular forms to automorphic
pseudodifferential operators, and both of the short exact sequences
\eqref{E:0b} and \eqref{E:6b} split.

The noncommutative multiplication operation in $\sdf$ defined by the
Leibniz rule determines the Lie bracket
\[ [\;\, ,\, ]^\pa: \sdf \times \sdf \to \sdf ,\]
on the same space given by
\begin{equation} \label{E:re}
[\Psi (z), \Phi (z)]^\pa = \Psi (z) \Phi (z) - \Phi (z) \Psi (z)
\end{equation}
for all $\Psi (z), \Phi (z) \in \sdf$, which provides $\sdf$ with a
structure of a complex Lie algebra.

Given $\gd_1, \gd_2 >0$ and pseudodifferential operators $\Psi (z) \in
\sdf_{-\gd_1}$ and $\Phi (z) \in \sdf_{-\gd_2}$ of the form
\[ \Psi (z) = \sum^\infty_{k =0} \psi_k (z) \pa^{-k -\gd_1}, \quad \Phi
(z) = \sum^\infty_{k =0} \phi_k (z) \pa^{-k -\gd_2} ,\]
we have
\[
\Psi (z) \Phi (z) = \sum^\infty_{k =0} \sum^\infty_{\ell =0}
\sum^\infty_{q =0} \bi {-k -\gd_1} {q} \psi_k (z) \phi_\ell^{(q)} (z)
\pa^{-k
  -\ell -q - \gd_1 -\gd_2} .
\]
Changing the indices from $k, \ell, q$ to $r, p, q$ with $r = k +\ell +q$
and $p = \ell +q$, we obtain
\[ \Psi (z) \Phi (z) = \sum^\infty_{r =0} \sum^r_{p =0}
\sum^p_{q =0} \bi {-r -\gd_1 +p} {q} \psi_{r-p} (z) \phi_{p -q}^{(q)} (z)
\pa^{-r - \gd_1 -\gd_2} .\]
Hence we see that
\begin{equation} \label{E:ws}
[\Psi (z), \Phi (z)]^\pa = \Psi (z) \Phi (z) - \Phi (z) \Psi (z) =
\sum^\infty_{r =0}  \eta (\Psi, \Phi)_r (z) \pa^{-r - \gd_1 -\gd_2} ,
\end{equation}
where
\begin{align*}
\eta (\Psi, \Phi)_r (z) = \sum^r_{p =0} \sum^p_{q =0} \biggl(
\bi {-r -\gd_1 +p} {q}& \psi_{r-p} (z) \phi_{p -q}^{(q)} (z)\\
&- \bi {-r -\gd_2 +p} {q} \phi_{r-p} (z) \psi_{p -q}^{(q)} (z) \biggr)
\end{align*}
for all $r \geq 0$.


If $r, p, q, \ga$ and $\gb$ are integers with $r \geq p \geq q$ and $\ga,
\gb >0$, then we set
\begin{equation} \label{E:xw}
\Xi^{r,p, q}_{\ga, \gb} = \frac {(r-p +\ga)! (r-p +q +\ga -1)!
(p-q +\gb)! (p-q +\gb -1)!} {q! (r + \ga + \gb)! (r + \ga + \gb -1)!} .
\end{equation}
We now consider formal power series $F(z,X) \in \mF [[X]]_{\gd_1}$ and
$G(z,X) \in \mF [[X]]_{\gd_2}$ given by
\begin{equation} \label{E:yf}
F (z, X) = \sum^\infty_{k=0} f_k (z) X^{k+\gd_1}, \quad  G (z, X) =
\sum^\infty_{k=0} g_k (z) X^{k+\gd_2} ,
\end{equation}
and define the associated formal power series $[F(z,X), G(z,X)]^X$ by
\begin{align} \label{E:mb}
[F (z,X), G(z,X)]^X &= \sum^\infty_{r =0} \sum^r_{p =0} \sum^p_{q =0}
\biggl( \Xi^{r,p,q}_{\gd_1 +\xi_1, \gd_2 +\xi_2} f_{r-p} (z)
  g^{(q)}_{p-q} (z)\\
&\hspace{1.5in} - \Xi^{r,p,q}_{\gd_2 +\xi_2, \gd_1 +\xi_1} g_{r-p} (z)
f^{(q)}_{p-q} (z) \biggr) X^{r + \gd_1 + \gd_2} .\notag
\end{align}

\begin{lem} \label{L:pa}
The formula \eqref{E:mb} determines a bilinear map
\[ [\cdot ,\cdot ]^X : \mF [[X]]_{\gd_1} \times \mF [[X]]_{\gd_2} \to
\mF [[X]]_{\gd_1 + \gd_2} \]
of formal power series satisfying
\begin{equation} \label{E:fn}
\mI_{\xi_1 + \xi_2}^\pa \bigl( [F (z,X), G (z,X)]^X \bigr) =
[\mI_{\xi_1}^\pa (F (z,X)), \mI_{\xi_2}^\pa (G (z,X))]^\pa
\end{equation}
for $F(z,X) \in \mF [[X]]_{\gd_1}$, $G(z,X) \in \mF [[X]]_{\gd_2}$ and
$\xi_1, \xi_2 \geq 0$.
\end{lem}

\begin{proof}
Let $F(z,X) \in \mF [[X]]_{\gd_1}$ and $G(z,X) \in \mF [[X]]_{\gd_2}$ be
as in \eqref{E:yf}.  Then by \eqref{E:nh} the pseudodifferential operators
$\mI_{\xi_1}^\pa (F (z,X))$ and $\mI_{\xi_2}^\pa (G (z,X))$ are given by
\begin{align*}
\mI_{\xi_1}^\pa (F (z,X)) &= \sum^\infty_{k=0} C_{k+\gd_1 +\xi_1} f_k (z)
\pa^{-k -\gd_1 -\xi_1},\\
\mI_{\xi_2}^\pa (G (z,X)) &= \sum^\infty_{k=0} C_{k+\gd_2 +\xi_2} g_k (z)
\pa^{-k -\gd_2 -\xi_2}
\end{align*}
for $\xi_1, \xi_2 \geq 0$.  From these relations and \eqref{E:ws} we
obtain
\begin{align*}
[\mI_{\xi_1}^\pa (F &(z,X)), \mI_{\xi_2}^\pa (G (z,X))]^\pa\\
&= \sum^\infty_{r =0} \sum^r_{p =0} \sum^p_{q =0} \biggl( \bi {-r -\gd_1
-\xi_1 +p} {q} C_{r -p+\gd_1 +\xi_1} C_{p-q +\gd_2 +\xi_2} f_{r-p} (z)
  g^{(q)}_{p-q} (z)\\
&\hspace{.5in} - \bi {-r -\gd_2 -\xi_2 +p} {q} C_{r -p+\gd_2 +\xi_2}
C_{p-q +\gd_1 +\xi_1}  g_{r-p} (z) f^{(q)}_{p-q} (z) \biggr) \pa^{-r -
2\gd - 2\xi} .
\end{align*}
Using this and \eqref{E:nm}, we have
\begin{align*}
\mI_\xi^X& [\mI_\xi^\pa (F (z,X)), \mI_\xi^\pa (G (z,X))]^\pa\\
&= \sum^\infty_{r =0} \sum^r_{p =0} \sum^p_{q =0} \biggl( \bi {-r -\gd_1
-\xi_1 +p} {q} C_{r -p+\gd_1 +\xi_1} C_{p-q +\gd_2 +\xi_2} f_{r-p} (z)
  g^{(q)}_{p-q} (z)\\
&\hspace{.2in} - \bi {-r -\gd_2 -\xi_2 +p} {q} C_{r -p+\gd_2 +\xi_2}
C_{p-q +\gd_1 +\xi_1}  g_{r-p} (z) f^{(q)}_{p-q} (z) \biggr) \frac {X^{r +
\gd_1 + \gd_2}} {C_{r + \gd_1 +\gd_2 +\xi_1 + \xi_2}} ,
\end{align*}
which can be easily seen to coincide with the right hand side of
\eqref{E:mb}.
\end{proof}

\begin{cor}
If $[ \cdot, \cdot ]^\pa$ is as in \eqref{E:re}, then we have
\begin{equation} \label{E:fz}
\mI_{\xi_1 + \xi_2}^X \bigl( [\Psi(z), \Phi (z)]^\pa \bigr) =
[\mI_{\xi_1}^X (\Psi (z)), \mI_{\xi_2}^X (\Phi (z))]^X
\end{equation}
for $\Psi (z) \in \sdf_{-\ve_1}$, $\Phi \in \sdf_{-\ve_2}$ and $\xi_1,
\xi_2 \geq 0$.
\end{cor}

\begin{proof}
Given $\Psi (z) \in \sdf_{-\ve_1}$ and $\Phi \in \sdf_{-\ve_2}$, using
\eqref{E:th} and \eqref{E:fn}, we have
\begin{align*}
[\Psi (z), \Phi (z)]^\pa &= [(\mI_{\xi_1}^\pa \circ \mI_{\xi_1}^X ) (\Psi
(z)), (\mI_{\xi_2}^\pa \circ \mI_{\xi_2}^X) (\Phi (z))]^\pa\\
&= \mI_{\xi_1 +\xi_2}^\pa [\mI_{\xi_1}^X (\Psi (z)), \mI_{\xi_2}^X (\Phi
(z))]^X.
\end{align*}
Thus we obtain \eqref{E:fz} by applying $\mI_{\xi_1 +\xi_2}^X$ to this
relation and using \eqref{E:th} again.
\end{proof}

For each $i \in \{ 1,2 \}$ we consider a quasimodular polynomial given by
\begin{equation} \label{E:wk}
F_i (z,X) = \sum^{m}_{r=0} f_{i, r} (z) X^r \in QP^m_{2\xi_i} (\G) ,
\end{equation}
and set
\begin{align} \label{E:gd}
[F_1 (z,X), F_2 (z,X)]^Q &= \sum^m_{r =0}
\sum^{m-r}_{p =0} \sum^p_{q =0} \frac {(r+p)! (m-p+q)!} {r!}\\
&\hspace{.3in} \times \biggl( \Xi^{m-r,p,q}_{\gd_1 +\xi_1, \gd_2 +\xi_2}
f_{1, r+p} (z) f^{(q)}_{2, m-p+q} (z) \notag\\
&\hspace{.9in} - \Xi^{m-r,p,q}_{\gd_2 +\xi_2, \gd_1 +\xi_1} f_{2, r+p} (z)
f^{(q)}_{1, m-p+q} (z) \biggr) X^{r} ,\notag
\end{align}
where the coefficients $\Xi^{*,*,*}_{*,*}$ are as in \eqref{E:xw}.
Let $\wh{\mF}_m$ be the graded complex vector space
\[ \wh{\mF}_m = \bigoplus_{\ell \geq 0} QP^m_{\ell} (\G) \]
for $m \geq 0$.  The next theorem shows that $\wh{\mF}_m$  has the
structure of a complex Lie algebra.

\begin{thm} \label{T:jr}
The bilinear map $[\cdot, \cdot]^Q$ given by \eqref{E:gd} is a Lie bracket
on the space $\wh{\mF}_m$ compatible with the Lie bracket $[\cdot,
\cdot]^X$ on Jacobi-like forms given by \eqref{E:mb}, meaning that the
diagram
\[\begin{CD}
\mJ_{2 (\xi_1 -m -\gd_1)} (\G)_{\gd_1}\times \mJ_{2 (\xi_2 -m -\gd_2)}
(\G)_{\gd_2} @> [\cdot, \cdot]^X >> \mJ_{2 (\xi_1 +\xi_2 -2m -\gd_1
-\gd_2)} (\G)_{\gd_1 +\gd_2}\\
@V (\Pi^{\gd_1}_m, \Pi^{\gd_1}_m) VV @VV \Pi^{\gd_1 +\gd_2}_m V\\
QP^m_{2\xi_1} (\G) \times QP^m_{2\xi_2} (\G)  @> [\cdot, \cdot]^Q
>> QP^m_{2\xi_1 +2\xi_2 -2m} (\G)
\end{CD}\]
is commutative.
\end{thm}

\begin{proof}
Let $F_i (z,X)$ with $1 \leq i \leq 2$ be as in \eqref{E:wk}, and assume
that
\[ \mL^{\gd_i}_{m, 2\xi_i} F_i (z, X) = \Phi_i (z,X) = \sum^\infty_{k=0}
\phi_{i, k} (z) X^{k + \gd_i} \in \mJ_{\gl_i -2m -2 \gd_i} (\G)_{\gd_i} \]
with $\Pi^{\gd_i}_m (\mL^{\gd_i}_{m, 2\xi_i} F_i) = F_i$, where
$\mL^{\gd_i}_{m, 2\xi_i}$ is as in \eqref{E:3p}. Then we have
\begin{equation} \label{E:rz}
\phi_{i, r} = (m -r)! f_{i, m -r}
\end{equation}
for each $r \in \{ 0, 1, \ldots m \}$.  From \eqref{E:mb} we see that
\begin{align*}
[\Phi_1 (z,X), \Phi_2 (z,X)]^X &= \sum^\infty_{r =0} \sum^r_{p =0}
\sum^p_{q =0} \biggl( \Xi^{r,p,q}_{\gd_1 +\xi_1, \gd_2 +\xi_2} \phi_{1,
r-p}
(z) \phi^{(q)}_{2, p-q} (z)\\
&\hspace{1.5in} - \Xi^{r,p,q}_{\gd_2 +\xi_2, \gd_1 +\xi_1} \phi_{2, r-p}
(z) \phi^{(q)}_{1, p-q} (z) \biggr) X^{r + \gd_1 + \gd_2} ,
\end{align*}
which is a Jacobi-like form belonging to $\mJ_{2\xi_1 + 2\xi_2}
(\G)_{\gd_1 +\gd_2}$.  Thus, using \eqref{E:rz}, we obtain
\begin{align*}
\Pi^{\gd_1 + \gd_2}_m [\Phi_1 (z,X), \Phi_2 (z,X)]^X &= \sum^m_{r =0}
\sum^{m-r}_{p =0} \sum^p_{q =0} \frac {1} {r!} \biggl(
\Xi^{m-r,p,q}_{\gd_1 +\xi_1, \gd_2 +\xi_2}
\phi_{1, m-r-p} (z) \phi^{(q)}_{2, p-q} (z)\\
&\hspace{.9in} - \Xi^{m-r,p,q}_{\gd_2 +\xi_2, \gd_1 +\xi_1} \phi_{2,
m-r-p} (z) \phi^{(q)}_{1, p-q} (z) \biggr) X^{r}\\
&= [F_1 (z,X), F_2 (z,X)]^Q ,
\end{align*}
and it is a quasimodular polynomial belonging to $QP^m_{2\xi_1 + 2\xi_2
-2m} (\G)$.  In particular, we have
\[ [F_1 (z,X), F_2 (z,X)]^Q = \Pi^{\gd_1 + \gd_2}_m
[\mL^{\gd_1}_{m, 2\xi_1} F_1 (z, X), \mL^{\gd_2}_{m, 2\xi_2} F_2 (z, X)]^X
.\]
We now consider an element
\[ \Psi_i (z,X) = \sum^\infty_{k=0} \psi_{i, k} (z) X^{k + \gd_i}
\in \mJ_{\gl_i -2m -2 \gd_i} (\G)_{\gd_i} \]
for each $i \in \{ 1, 2\}$, and let
\[ G_i (z, X) = \Pi^{\gd_i}_m \Psi_i (z,X) .\]
If we set
\[ [\mL^{\gd_1}_{m, 2\xi_1} G_1 (z, X), \mL^{\gd_2}_{m, 2\xi_2} G_2 (z,
X)]^X = \sum^\infty_{k=0} \ga_k (z) X^{k + \gd_i +\gd_i} ,\]
\[ [\Psi_1 (z, X), \Psi_2 (z, X)]^X = \sum^\infty_{k=0} \gb_k (z)
X^{k + \gd_i +\gd_i} ,\]
then from \eqref{E:mb} we see that $\ga_k = \gb_k$ for each $k \in \{ 0,
1, \ldots, m \}$; hence we obtain
\[ \Pi^{\gd_1 + \gd_2}_m [\mL^{\gd_1}_{m, 2\xi_1} G_1 (z, X),
\mL^{\gd_2}_{m, 2\xi_2} G_2 (z, X)]^X = \Pi^{\gd_1 + \gd_2}_m [\Psi_1 (z,
X), \Psi_2 (z, X)]^X .\]
Thus we have
\[ [\Pi^{\gd_1}_m \Psi_1 (z,X), \Pi^{\gd_2}_m \Psi_2 (z,X)]^Q
= \Pi^{\gd_1 + \gd_2}_m [\Psi_1 (z, X), \Psi_2 (z, X)]^X ,\]
which proves the theorem.
\end{proof}

\begin{rem}
From \eqref{E:gd} we see that the coefficient of $X^m$ in $[F_1 (z,X), F_2
(z,X)]^Q$ is given by
\[ \fS_m [F_1 (z,X), F_2 (z,X)]^Q = \Xi^{0,0,0}_{\gd_1 +\xi_1, \gd_2 +\xi_2}
f_{1,m} (z) f_{2,m} (z) - \Xi^{0,0,0}_{\gd_2 +\xi_2, \gd_1 +\xi_1} f_{2,m}
(z) f_{1,m} (z) = 0 ;\]
hence it follows that
\[ [F_1 (z,X), F_2 (z,X)]^Q \in QP^{m-1}_{2\xi_1 +2\xi_2 -2m} (\G) .\]
On the other hand, the coefficient of $X^{m-1}$ is equal to
\begin{align*}
\fS_{m-1}& [F_1 (z,X), F_2 (z,X)]^Q
\\&= m! \Bigl( \Xi^{1,0,0}_{\gd_1 +\xi_1, \gd_2 +\xi_2}
f_{1,m-1} (z) f_{2,m} (z) - \Xi^{1,0,0}_{\gd_2 +\xi_2, \gd_1 +\xi_1} f_{2,m-1}
(z) f_{1,m} (z) \Bigr)\\
&\hspace{.3in} + m! \Bigl( \Xi^{1,1,0}_{\gd_1 +\xi_1, \gd_2 +\xi_2}
f_{1,m} (z) f_{2,m -1} (z) - \Xi^{1,1,0}_{\gd_2 +\xi_2, \gd_1 +\xi_1} f_{2,m}
(z) f_{1,m-1} (z) \Bigr)\\
&\hspace{.6in} + m(m!) \Bigl( \Xi^{1,1,1}_{\gd_1 +\xi_1, \gd_2 +\xi_2}
f_{1,m} (z) f'_{2,m} (z) - \Xi^{1,1,1}_{\gd_2 +\xi_2, \gd_1 +\xi_1} f_{2,m}
(z) f'_{1,m} (z) \Bigr)\\
&=  m(m!) \Bigl( \Xi^{1,1,1}_{\gd_1 +\xi_1, \gd_2 +\xi_2}
f_{1,m} (z) f'_{2,m} (z) - \Xi^{1,1,1}_{\gd_2 +\xi_2, \gd_1 +\xi_1} f_{2,m}
(z) f'_{1,m} (z) \Bigr),
\end{align*}
and it is a modular form belonging to $M_{2\xi_1 +2\xi_2 -4m +2} (\G)$ by
\eqref{E:hh}.
\end{rem}

\begin{cor}
The Lie bracket $[\cdot, \cdot]^Q$ on $\wh{\mF}_m$ given by \eqref{E:gd}
is compatible with the Lie bracket $[\cdot, \cdot]^\pa$ on $\sdf$ given by
\eqref{E:re}, so that
\[\begin{CD}
\sdf^\G_{-\xi_1 +m} \times \sdf^\G_{-\xi_2 +m} @> [\cdot, \cdot]^\pa >>
\sdf^\G_{-\xi_1 -\xi_2 +2m}\\
@V (\PPi^{m -\xi_1}_m, \PPi^{m -\xi_2}_m) VV @VV \PPi^{2m -\xi_1 -\xi_2}_m V\\
QP^m_{2\xi_1} (\G) \times QP^m_{2\xi_2} (\G)  @> [\cdot, \cdot]^Q
>> QP^m_{2\xi_1 +2\xi_2 -2m} (\G)
\end{CD}\]
is commutative, where $\PPi^*_m$ is as in \eqref{E:z8}.
\end{cor}

\begin{proof}
This follows from \eqref{E:17}, \eqref{E:fz} and Theorem \ref{T:jr},
\end{proof}

\section{\bf{Rankin-Cohen brackets on quasimodular forms}}

Rankin-Cohen brackets for modular forms are well-known (cf.\ \cite{CM97}),
and similar brackets for Jacobi-like forms were introduced in (\cite{C1}, \cite{C2}, \cite{CL07})
by using the heat operator.  Rankin-Cohen brackets on quasimodular forms
were also studied by Martin and Royer in \cite{MR07} for congruence
subgroups of $SL(2, \bZ)$.  In this section we construct Rankin-Cohen
brackets on quasimodular forms for more general discrete subgroups of
$SL(2, \bR)$ that are compatible with Rankin-Cohen brackets on Jacobi-like
forms.

Given $\mu \in \bR$, we consider the formal differential operator $\fL_\mu$
on $\mF [[X]]$ given by
\begin{equation} \label{E:55}
\fL_\mu = \frac {\pa} {\pa z} - \mu \frac {\pa} {\pa X} - X
\frac {\pa^2} {\pa X^2} ,
\end{equation}
which may be regarded as the radial heat operator in some
sense (cf.\ \cite{L09a}).

\begin{prop}
Given $\mu \in \bR$, $\gl \in \bZ$ and a formal power series $\Phi (z,X)
\in \mF [[X]]$, we have
\begin{equation} \label{E:5t}
(\fL_\mu (\Phi)|_{\gl+2}^J \g) (z,X) = \fL_\mu (\Phi|_{\gl}^J \g)
(z,X) + (\gl-\mu  ) \fK (\g,z) (\Phi|_{\gl}^{J} \g) (z,X)
\end{equation}
for all $\g \in SL(2, \bR)$.
\end{prop}

\begin{proof}
Let $\g$ be an element of $SL(2,\mbb R)$ whose $(2,1)$-entry is $c$, so
that
\[ \frac {\pa} {\pa z} \fJ (\g,z) = c = \fJ (\g,z) \fK (\g,z) ,\]
where $\fJ (\g,z)$ and $\fK (\g,z)$ are as in \eqref{E:kz}.  Given a
formal power series $\Phi (z,X) \in \mF [[X]]$, using \eqref{E:yt}, we see
that
\begin{align*}
\frac {\pa} {\pa z} (\Phi \mid^J_\gl \g) (z,X)
&= -\gl c \fJ (\g,z)^{-\gl-1}
e^{-\fK (\g,z) X} \Phi (\g z, \fJ (\g,z)^{-2} X)\\
&\hspace{.3in} + \fJ (\g,z)^{-\gl} \fK (\g,z)^2 X e^{-\fK (\g,z) X} \Phi
(\g z, \fJ (\g,z)^{-2} X)\\
&\hspace{.5in} + \fJ (\g,z)^{-\gl} e^{-\fK (\g,z) X} \fJ (\g,z)^{-2} \frac
{\pa \Phi} {\pa z} (\g z, \fJ (\g,z)^{-2} X)\\
&\hspace{.7in} + \fJ (\g,z)^{-\gl} e^{-\fK (\g,z) X} (-2c) \fJ (\g,z)^{-3}
X \frac {\pa \Phi} {\pa X} (\g z, \fJ (\g,z)^{-2} X),
\end{align*}
\begin{align*}
\frac {\pa} {\pa X} (\Phi \mid^J_\gl \g) (z,X)
&= -\fJ (\g,z)^{-\gl} \fK
(\g,z) e^{-\fK (\g,z) X} \Phi (\g z, \fJ (\g,z)^{-2} X)\\
&\hspace{.6in} + \fJ (\g,z)^{-\gl} e^{-\fK (\g,z) X}  \fJ (\g,z)^{-2}\frac
{\pa \Phi} {\pa X} (\g z, \fJ (\g,z)^{-2} X)\\
& = \fJ (\g,z)^{-\gl-2} e^{-\fK (\g,z) X}\\
&\hspace{.3in} \times \biggl( -c \fJ (\g,z) \Phi (\g z, \fJ (\g,z)^{-2} X)
+ \frac {\pa \Phi} {\pa X} (\g z, \fJ (\g,z)^{-2} X) \biggr),
\end{align*}
\begin{align*}
\frac {\pa^2} {\pa X^2} (\Phi \mid^J_\gl \g) (z,X)
&= \fJ (\g,z)^{-\gl-2}
  e^{-\fK (\g,z) X} \biggl( c^2 \Phi (\g z, \fJ (\g,z)^{-2} X)\\
& - 2\fK (\g,z) \frac {\pa \Phi} {\pa X} (\g z, \fJ (\g,z)^{-2} X) +  \fJ
  (\g,z)^{-2} \frac {\pa^2 \Phi} {\pa X^2} (\g z, \fJ (\g,z)^{-2} X)
  \biggr) .
\end{align*}
From these relations and \eqref{E:55}, we obtain
\begin{align*}
\fL_\mu (\Phi \mid^J_\gl \g) (z,X)
&= \frac {\pa} {\pa z} (\Phi \mid^J_\gl \g) (z,X) - \mu \frac {\pa} {\pa
X} (\Phi \mid^J_\gl \g) (z,X) - X \frac {\pa^2} {\pa X^2} (\Phi \mid^J_\gl
\g) (z,X)\\
& = \fJ (\g,z)^{-\gl-2} e^{-\fK (\g,z) X}\\
&\hspace{.2in} \times \biggl( - \gl c \fJ (\g,z) \Phi + \fJ (\g,z)^2 \fK
(\g,z)^2 X \Phi + \frac {\pa \Phi} {\pa z}- 2c \fJ (\g,z)^{-1} X \frac {\pa
  \Phi} {\pa X}\\
&\hspace{1.0in} +\mu c \fJ (\g,z) \Phi -\mu \frac {\pa \Phi} {\pa X} -c^2 X
\Phi\\
&\hspace{.8in} + 2 \fK
(\g,z) X \frac {\pa \Phi} {\pa X} - \fJ (\g,z)^{-2} X \frac {\pa^2 \Phi}
{\pa X^2} \biggr) (\g z, \fJ (\g,z)^{-2} X)\\
&= \fJ (\g,z)^{-\gl-2} e^{-\fK (\g,z) X} \biggl( \frac {\pa \Phi}
{\pa z} -\mu \frac {\pa \Phi} {\pa X} - \fJ (\g,z)^{-2} X \frac {\pa^2 \Phi}
{\pa X^2} \biggr) (\g z, \fJ (\g,z)^{-2} X)\\
&\hspace{1.0in} + (\mu -\gl) c \fJ (\g,z) \Phi (\g z, \fJ (\g,z)^{-2} X),
\end{align*}
where we used the relation $c = \fJ (\g,z) \fK (\g,z)$.  Thus it follows
that
\begin{align*}
(\fL_\mu (\Phi \mid^J_\gl \g)) (z,X) &= \fJ (\g,z)^{-\gl-2} e^{-\fK
(\g,z) X} (\fL_\mu \Phi) (\g z, \fJ (\g,z)^{-2} X)\\
&\hspace{1.0in} + (\mu -\gl) c \fJ (\g,z)^{-\gl-1} e^{-\fK
(\g,z) X} \Phi (\g z, \fJ (\g,z)^{-2} X)\\
&= (\fL_\mu (\Phi) \mid^J_{\gl+2} \g) (z,X) + (\mu -\gl) \fK (\g,z) (\Phi
\mid^J_{\gl} \g) (z,X),
\end{align*}
which verifies \eqref{E:5t}.
\end{proof}

Let $\G$ be a discrete subgroup of $SL(2, \bR)$.  If $\mu = \gl$, then
\eqref{E:5t} can be written in the form
\[ (\fL_{\gl}(\Phi)|_{\gl+2}^J \g) (z,X) = \fL_{\gl} (\Phi|_{\gl}^J
\g) (z,X) .\]
Thus we see that
\[ \fL_\gl (\mJ_\gl (\G)) \subset \mJ_{\gl +2} (\G) .\]
If $\mu \neq \gl$, however, the operator $\fL_\mu$ does not carry
Jacobi-like forms to Jacobi-like forms.

If $\Phi (z,X) = \sum^\infty_{k=0} \phi_k (z) X^{k +\gd} \in \mF
[[X]]_\gd$ with $\gd \geq 0$, then it can be shown that
\[ (\fL^\ell_\mu \Phi) (z,X) = \sum^\infty_{k=0} [\Phi]^{\ell, k}_{\mu,
  \gd} (z) X^{k +\gd} ,\]
  (Note that $\fL^\ell_\mu $ denotes the composition of
  $\ell$-copies of $\fL_{\mu}.$)
where
\begin{align} \label{E:7f}
[\Phi]^{\ell, k}_{\mu, \gd} (z) =  \sum^{\ell  }_{j=1} &
(-1)^{\ell -j} \bi {\ell} {j}\\
&\times\frac{(k+\delta+\ell-j)!(k+\delta+\mu+\ell-j-1)!}
{(k+\delta)!(k+\delta+\mu-1)!}
 \phi^{(j)}_{k+\ell-j} (z) \notag
\end{align}
for all $z \in \mH$ and $k \geq 0$.

We now introduce bilinear maps on the space $\mF [[X]]$ of formal power
series by using the operators of the form $\fL_\mu$ with $\mu \in \frac12
\bZ$.  Given a nonnegative integer $n$, elements $\mu_1, \mu_2 \in \frac12
\bZ$ and formal power series
\[ \Phi_1 (z,X) \in \mJ_{\gl_1}(\G)_{\gd_1}, \quad \Phi_2 (z,X) \in
 \mJ_{\gl_2}(\G)_{\gd_2} ,\]
we define the associated formal power series $[\Phi_1,
\Phi_2]^J_{\mu_1,
  \mu_2, n} (z,X) \in \mF [[X]]_{\gd_1 + \gd_2}$ by
\begin{equation} \label{E:8i}
[\Phi_1, \Phi_2]^J_{\mu_1, \mu_2, n} (z,X) = \sum^\infty_{u=0} \xi_{\mu_1,
  \mu_2, n, u} (z) X^{u + \gd_1 + \gd_2}
\end{equation}
with
\begin{align} \label{E:82}
\xi_{\mu_1, \mu_2, n, u} (z) &= \sum^{u }_{t=0} \sum^n_{\ell =0} (-1)^\ell \bi {n +
\gl_1 -\mu_1 -1} {n-\ell} \bi {n + \gl_2 -\mu_2 -1} {\ell}\\
&\hspace{1.9in} \times [\Phi_1]^{\ell, t}_{\mu_1, \gd_1} (z) [\Phi_2]^{n
  -\ell, u-t}_{\mu_2, \gd_2} (z) \notag
\end{align}
for all $z \in \mH$ and $u \geq 0$, where $[\Phi_1]^{\ell,
t}_{\mu_1, \gd_1}$ and $[\Phi_2]^{n -\ell, u -t}_{\mu_2, \gd_2}$
are as in \eqref{E:7f}.  Here we note that the binomial coefficients of the
form $\bi r\ell$ with $r \in \frac 12 \bZ$ is given by
\[ \bi r\ell = \frac {r!} {(r-\ell)! \ell!} = \frac {\BG (r+1)} {\BG
(r-\ell+1) \BG(\ell+1) } ,\] where $\BG$ is the Gamma function.

\begin{prop} \label{P:jw}
If $\Phi_1 (z,X)$ and $\Phi_2 (z,X)$ are Jacobi-like forms with
\[ \Phi_1 (z,X) \in \mJ_{\gl_1} (\G)_{\gd_1}, \quad \Phi_2 (z,X) \in
\mJ_{\gl_2} (\G)_{\gd_2} ,\]
then $[\Phi_1, \Phi_2]^J_{\mu_1, \mu_2, n} (z,X)$ is a Jacobi-like form
belonging to $\mJ_{\gl_1 +\gl_2 +2n} (\G)_{\gd_1 +\gd_2}$.
\end{prop}

\begin{proof}
This was proved in \cite{CL07} in the case where $\mu_1= \mu_2 =
1/2$. The general case can be proved in a similar manner.
\end{proof}

From Proposition \ref{P:jw} we obtain the bilinear map
\begin{equation} \label{E:83}
[\, ,\,]^J_{\mu_1, \mu_2, n}: \mJ_{\gl_1} (\G)_{\gd_1} \times
\mJ_{\gl_2} (\G)_{\gd_2} \to \mJ_{\gl_1 +\gl_2+2n} (\G)_{\gd_1
+\gd_2} ,
\end{equation}
which may be regarded as the $n$-th Rankin-Cohen bracket for Jacobi-like
forms (see \cite{CL07}).

Given integers $\gl_1$, $\gl_2$, nonnegative integers $\gd_1, \gd_2, m_1,
m_2, n$, and quasimodular polynomials
\begin{equation} \label{E:oo}
F_1 (z,X) \in QP^{m_1}_{\gl_1+2(m_1+   \gd_1)} (\G) \subset \mF_{m_1}
[X] ,\quad F_2 (z,X) \in QP^{m_2}_{ \gl_2+2(m_2+ \gd_2)} (\G) \subset
\mF_{m_2} [X] ,
\end{equation}
using the lifting map in \eqref{E:3p}, we obtain the Jacobi-like forms
\begin{equation} \label{E:x6}
\mL^{\gd_1}_{m_1, \gl_1 +2(m_1 +\gd_1)} F_1 (z,X) \in \mJ_{\gl_1} (\G)_{\gd_1}
, \quad
\mL^{\gd_2}_{m_2, \gl_2 +2(m_2 +\gd_2)} F_2 (z,X) \in \mJ_{\gl_2}
(\G)_{\gd_2} .
\end{equation}
If $\mu_1, \mu_2 \in \frac12 \bZ$, we define the bilinear map
\[ [\, ,\,]^{\gl_1, \gl_2, P}_{\gd_1, \gd_2, \mu_1, \mu_2, n}:
QP^{m_1}_{\gl_1+2(m_1+ \gd_1)} (\G) \times QP^{m_2}_{ \gl_2+2(m_2+ \gd_2)}
(\G) \to \mF_{m_1 +m_2} [X] \]
of polynomials by
\begin{align} \label{E:36}
[F_1 , F_2& ]^{\gl_1, \gl_2, P}_{\gd_1, \gd_2, \mu_1, \mu_2, n} (z,X)\\
&= \sum^{m_1 + m_2}_{r=0} \sum^n_{\ell =0} \sum^{m_1 + m_2
-r}_{t=0} \frac {(-1)^\ell} {r!} \bi {n +\gl_1 -\mu_1 -1}
{n-\ell} \bi {n +\gl_2 -\mu_2-1}
{\ell} \notag\\
&\hspace{1.2in} \times
[\mL^{\gd_1}_{m_1, \gl_1 +2(m_1 +\gd_1)} F_1]^{\ell, t}_{\mu_1, \gd_1} (z) \notag\\
&\hspace{1.7in} \times [\mL^{\gd_2}_{m_2, \mu_2 +2(m_2 +\gd_2)}
F_2]^{n
  -\ell, m_1 + m_2 -r -t}_{\mu_2, \gd_2} (z) X^r ,\notag
\end{align}
where the square brackets on the right hand side are as in \eqref{E:7f}.

\begin{prop} \label{P:rc}
The formula \eqref{E:36} determines a bilinear map
\begin{equation} \label{E:rk}
[\, ,\,]^{\gl_1, \gl_2, P}_{\gd_1, \gd_2, \mu_1, \mu_2, n}: QP^{m_1}_{\gl_1+2(m_1+
  \gd_1)} (\G) \times QP^{m_2}_{ \gl_2+2(m_2+ \gd_2)} (\G) \to
QP^{m_1+m_2}_{\gl_1 +\gl_2+2(m_1+m_2 +n+\gd_1+\gd_2)} (\G)
\end{equation}
of quasimodular polynomials.
\end{prop}

\begin{proof}
Let $F_1 (z,X)$ and $F_2 (z,X)$ be quasimodular polynomials as in
\eqref{E:oo}.  If $[\, ,\,]^J_n$ is as in \eqref{E:83}, then from
\eqref{E:8i} and \eqref{E:82} we obtain
\begin{align} \label{E:9i}
[\mL^{\gd_1}_{m_1, \gl_1 +2(m_1 +\gd_1)}& F_1, \mL^{\gd_2}_{m_2, \gl_2
  +2(m_2 +\gd_2)}  F_2]^J_{\mu_1, \mu_2, n} (z,X)\\
&\hspace{.4in}= \sum^\infty_{u=0} \wt{\xi}_u (z) X^{u +
  \gd_1 + \gd_2} \in \mJ_{\gl_1 +\gl_2} (\G)_{\gd_1 +\gd_2} ,\notag
\end{align}
where
\begin{align} \label{E:92}
\wt{\xi}_u (z) = \sum^{u }_{t=0} \sum^n_{\ell =0} (-1)^\ell & \bi
{n + \gl_1 -\mu_1 -1} {n-\ell} \bi {n + \gl_2 -\mu_2  -1} {\ell}\\
&\times [\mL^{\gd_1}_{m_1, \gl_1
  +2(m_1 +\gd_1)} F_1]^{\ell, t}_{\mu_1, \gd_1} (z) [\mL^{\gd_2}_{m_2, \gl_2
  +2(m_2 +\gd_2)} F_2]^{n -\ell, u -t}_{\mu_2, \gd_2} (z). \notag
\end{align}
Using \eqref{E:nn}, \eqref{E:36}. \eqref{E:9i} and \eqref{E:92}, we see
that
\begin{align*}
[F_1 , F_2]^{\gl_1, \gl_2, P}_{\gd_1, \gd_2, \mu_1, \mu_2, n} (z,X)
&= \sum^{m_1  + m_2}_{r=0} \frac 1{r!} \wt{\xi}_{m_1 + m_2 -r} (z) X^r\\
&= \Pi^{\gd_1 +\gd_2}_{m_1 + m_2} ([\mL^{\gd_1}_{m_1, \gl_1 +2(m_1 +\gd_1)}
F_1 , \mL^{\gd_2}_{m_2, \gl_2  +2(m_2 +\gd_2)}  F_2]^J_{\mu_1, \mu_2, n} (z,X)) ,
\end{align*}
which is a quasimodular polynomial belonging to
\[ QP^{m_1 +m_2}_{\gl_1+\gl_2 +2(m_1+m_2+n+\gd_1+\gd_2)} (\G) ;\]
hence the theorem follows.
\end{proof}

The bilinear map \eqref{E:rk} for quasimodular polynomials determines a
corresponding map for quasimodular forms, which may be regarded as a
Rankin-Cohen bracket.  We now consider a variation of such a bilinear map.
Given a nonnegative integer $m$ and quasimodular polynomials $F_1 (z,X)$
and $F_2 (z,X)$ as in \eqref{E:oo}, by applying the map $\Pi^{\gd_1
  +\gd_2}_m$ to the left hand side of \eqref{E:9i} we obtain
\begin{align} \label{E:ge}
\Pi^{\gd_1 +\gd_2}_m ([ & \mL^{\gd_1}_{m_1, \gl_1 +2(m_1 +\gd_1)}
F_1 , \mL^{\gd_2}_{m_2, \gl_2  +2(m_2 +\gd_2)}  F_2]^J_{\mu_1, \mu_2, n} (z,X))\\
&= \sum^m_{r=0} \frac 1{r!} \wt{\xi}_{m -r} (z) X^r \notag\\
&= \sum^m_{r=0} \sum^n_{\ell =0} \sum^{m -r+n}_{t=0} \frac
{(-1)^\ell} {r!} \bi {n +\gl_1 -\mu_1 -1} {n-\ell} \bi {n +\gl_2 -\mu_2 -1}
{\ell}
\notag\\
&\hspace{1.4in} \times
[\mL^{\gd_1}_{m_1, \gl_1 +2(m_1 +\gd_1)} F_1]^{\ell, t}_{\mu_1, \gd_1} (z) \notag\\
&\hspace{1.8in} \times [\mL^{\gd_2}_{m_2, \gl_2 +2(m_2 +\gd_2)} F_2]^{n
  -\ell, m -r -t}_{\mu_2, \gd_2} (z) X^r , \notag
\end{align}
which is a quasimodular polynomial belonging to
\[ QP^m_{\gl_1+\gl_2 +2(m +n+\gd_1+\gd_2)} (\G) .\]
Thus, if we set
\[ [[ F_1, F_2 ]]^{m}_{n} (z,X) = \Pi^{\gd_1 +\gd_2}_m ([\mL^{\gd_1}_{m_1, \gl_1
  +2(m_1 +\gd_1)} F_1 , \mL^{\gd_2}_{m_2, \gl_2  +2(m_2 +\gd_2)}
F_2]^J_{\mu_1, \mu_2, n} (z,X)) ,\]
we obtain the bilinear map
\begin{equation} \label{E:b3}
[[\, ,\,]]^{m}_{n} : QP^{m_1}_{\gl_1+2(m_1+   \gd_1)} (\G) \times
QP^{m_2}_{
  \gl_2+2(m_2+ \gd_2)} (\G) \to QP^m_{\gl_1+\gl_2 +2(m +n+\gd_1+\gd_2)}
(\G) .
\end{equation}
If $G (z,X) = \sum^q_{k=0} g_k (z) X^k  \in \mF_q [X]$ with $q \geq 0$, we
set
\begin{align} \label{E:4f}
[[G]]^{\ell, k}_{\mu, \gd} (z) = \sum_{j=0}^{\ell} (q-k-j)!  &
(-1)^{\ell -j} \bi {\ell} {j}  \\
&\times \frac{(k  +\gd+\ell-j)! (k
+\gd+\mu-\ell-j-1)!}{(k+\gd)!(k+\gd+\mu-1)!}
 g^{(j)}_{k+\ell-j} (z) \notag
\end{align}
for $z \in \mH$, $\gd \geq 0$ and $0 \leq k \leq q$.

\begin{prop}
Let $F_1 (z,X)$ and $F_2 (z,X)$ be as in Proposition \ref{P:rc}, and let
$m$ be a nonnegative integer with $m +2n \leq \min \{ m_1, m_2 \}$.  Then
the bilinear map \eqref{E:b3} is given by
\begin{align} \label{E:pp}
[[F_1 , F_2]]^{m}_{n} (z,X) &= \sum^{m_1 +
  m_2}_{r=0} \sum^n_{\ell =0} \sum^{m_1 + m_2 -r+n}_{t=0} \frac
{(-1)^\ell} {r!} \bi {n +\gl_1 -\mu_1 -1} {n-\ell} \bi {n +\gl_2 -\mu_2-1}
{\ell}\\
&\hspace{1.8in} \times [[F_1]]^{\ell, t}_{\mu, \gd_1} (z) [[F_2]]^{n
-\ell, m -r -t}_{\mu, \gd_2} (z) X^r, \notag
\end{align}
and the diagram
\begin{equation} \label{E:qq}
\begin{CD}
\mJ_{\gl_1} (\G)_{\gd_1} \times \mJ_{\gl_2}(\G)_{\gd_2}
@>  [\, , \, ]^J_{\mu_1, \mu_2, n} >>
  \mJ_{\gl_1+\gl_2+2n}(\G)_{\gd_1+\gd_2 }\\
@V (\Pi_{m_1}^{\gd_1}, \Pi_{m_2}^{\gd_2})
VV @VV  \Pi_{m_1+m_2}^{\gd_1+\gd_2} V\\
QP^{m_1}_{\gl_1+2(m_1+ \gd_1)}(\G) \times QP^{m_2}_
{\gl_2+2(m_2+ \gd_2)}(\G) @> [[\, ,\,]]_m >> QP^m_{\gl_1+\gl_2 +2(m
  +n +\gd_1+\gd_2)} (\G)
\end{CD}
\end{equation}
is commutative.
\end{prop}

\begin{proof}
If $m +2n \leq \min \{ m_1, m_2 \}$, then we see from \eqref{E:7f} that the
functions
\[ [\mL^{\gd_1}_{m_1, \gl_1 +2(m_1 +\gd_1)} F_1]^{\ell, t}_{\mu_1, \gd_1}
(z), \quad [\mL^{\gd_2}_{m_2, \gl_2 +2(m_2 +\gd_2)} F_2]^{n
  -\ell, m -r -t}_{\mu_2, \gd_2} (z) \]
in \eqref{E:ge} involve coefficients of $X^j$ only for $j \leq  \min \{
m_1, m_2 \}$ in the Jacobi-like forms
\[ \mL^{\gd_1}_{m_1, \gl_1 +2(m_1 +\gd_1)} F_1 (z,X), \quad \mL^{\gd_2}_{m_2,
  \gl_2 +2(m_2 +\gd_2)} F_2 (z,X) . \]
From this and the relations
\[ \Pi^{\gd_1}_{m_1} \mL^{\gd_1}_{m_1, \gl_1 +2(m_1 +\gd_1)} F_1 (z,X) =
F_1 (z,X), \quad \Pi^{\gd_2}_{m_2} \mL^{\gd_2}_{m_2,
  \gl_2 +2(m_2 +\gd_2)} F_2 (z,X) = F_2 (z,X) ,\]
it follows that
\[ [\mL^{\gd_1}_{m_1, \gl_1 +2(m_1 +\gd_1)} F_1]^{\ell, t}_{\mu_1, \gd_1} =
[[F_1]]^{\ell, t}_{\mu_1, \gd_1} ,\]
\[ [\mL^{\gd_2}_{m_2, \gl_2 +2(m_2 +\gd_2)} F_2]^{n -\ell, m -r
  -t}_{\mu_2, \gd_2} = [[F_2]]^{n -\ell, m -r -t}_{\mu_2, \gd_2} ;\]
hence we obtain \eqref{E:pp}.  The commutativity of the diagram
\eqref{E:qq} also follows from the above observations.
\end{proof}

\section{\bf{Quasimodular forms of half-integral weight}}

In this section we modify the definitions of Jacobi-like forms and
quasimodular forms to include the half-integral weights.  We also describe
Hecke operators on spaces of those forms.  Throughout this section we
assume that $\G = \G_0 (4N)$ for some positive integer $N$.

Let $\gt(z)$ be the theta series given by
\[ \gt (z) = \sum^\infty_{n= -\infty} e^{ 2\pi i n^2 z} \]
for $z \in \mH$, and set
\begin{equation} \label{hE:kz}
\fj (\g,z) = \frac {\gt (\g z)} {\gt (z)}, \quad \fk (\g,z) = 2 \fj
(\g,z)^{-1} \frac {d} {dz} \fj (\g,z)
\end{equation}
for all $\g \in \G$ and $z \in \mH$.  If $\g = \sm a&b \\ c&d \esm$, it is
known that
\[ \fj (\g, z)^2 = \Bigl( \frac {-1}d \Bigr) (cz+d) ,\]
where $( \frac {\;\cdot\;} {\; \cdot \;})$ denotes the Legendre symbol.
Furthermore, the resulting maps $\fj, \fk: \G \times \mH \to \bC$
satisfy
\begin{equation} \label{hE:kk}
\fj (\g \g', z) = \fj (\g, \g' z)\fj (\g', z), \quad \fk (\g \g',z) = \fk
(\g',z) + \fj (\g', z)^{-2} \fk (\g, \g' z)
\end{equation}
for all $\g, \g' \in \G$ and $z \in \mH$.

We recall that $\mF$ denotes the ring of holomorphic functions on $\mH$ and
$\mF [[X]]$ is the complex algebra of formal power series in $X$ with
coefficients in $\mF$.  From now on we assume that $\gl$ is a half integer,
so that $2\gl$ is an odd integer.  Given elements $f \in \mF$, $\Phi (z,X)
\in \mF [[X]]$, and $\g \in \G$, we set
\begin{equation} \label{hE:xt}
(f \mid_\gl \g) (z) = \fj (\g, z)^{-2\gl} f (z)
\end{equation}
\begin{equation} \label{hE:yt}
(\Phi \mid^J_\gl \g) (z,X) = \fj (\g, z)^{-2\gl} e^{- \fk (\g,
z) X} \Phi (\g z, \fj (\g, z)^{-2} X)
\end{equation}
for all $z \in \mH$.  If $\g'$ is another element of $\G$,
then by using \eqref{hE:kk} it can be shown that
\[ f \mid_\gl (\g \g') = (f \mid_\gl \g) \mid_\gl  \g', \quad \Phi
\mid^J_\gl (\g \g') = (\Phi \mid^J_\gl \g) \mid^J_\gl  \g' ,\]
so that the operations $\mid_\gl$ and $\mid^J_\gl$ determine right
actions of $\G$ on $\mF$ and $\mF [[X]]$, respectively.

Given a nonnegative integer $m$, let $\mF_m [X]$ be the complex algebra of
polynomials in $X$ over $\mF$ of degree at most $m$ as before.  If $\g \in
\G$ and $F (z,X) \in \mF_m [X]$, we set
\begin{equation} \label{hE:js}
(F \md_\gl \g) (z, X) = \fj (\g, z)^{-2\gl} F (\g z, \fj (\g,
z)^2 ( X - \fk (\g, z)))
\end{equation}%
for all $z \in \mH$.  Then this formula determines a right action $\md_\gl$
of $\G$ on $\mF_m [X]$.  Let $\chi$ be a character on $\G$.

\begin{dfn}
Given a half integer $\gl$ and a nonnegative integer $m$, an
element $f \in \mF$ is a {\em quasimodular form for $\G$ of weight
$\gl$ and depth at most $m$ with character $\chi$\/} if there are
functions $f_0, \ldots, f_m \in \mF$ such that
\begin{equation} \label{hE:qz}
(f \mid_{\gl} \g) (z) = \sum^m_{r=0} \chi(\g) f_r (z) \fk (\g, z)^r
\end{equation}
for all $z \in \mH$ and $\g \in \G$, where $\fk (\g, z)$ is as in
\eqref{hE:kz} and $\mid_{\gl}$ is the operation in
\eqref{hE:xt}. We denote by $QM^m_{\gl}(\G, \chi)$ the space of
quasimodular forms for $\G$ of weight $\gl$ and depth $m$ with
character $\chi$.
\end{dfn}

If we denote by $M_{\gl} (\G, \chi)$ the space of modular forms for $\G$
of weight $\gl$ and character $\chi$, then we see that
\[ QM^0_{\gl} (\G, \chi) = M_{\gl} (\G, \chi). \]
We also note that $f = f_0$ if $f$ satisfies \eqref{hE:qz}.

Let $f \in \mF$ be a quasimodular form belonging to $QM^m_{\gl}(\G, \chi)$
and satisfying \eqref{hE:qz}.  Then we define the corresponding
polynomial $(\mQ_{\gl}^m f) (z,X) \in \mF_m [X]$ by
\begin{equation} \label{hE:tp}
(\mQ_{\gl}^m f) (z,X) = \sum^m_{r=0} f_r (z) X^r,
\end{equation}
so that we obtain the linear map
\[ \mQ_{\gl}^m: QM^m_{\gl}(\G, \chi) \to \mF_m [X] \]
for each of nonnegative integers   $m.$

\begin{dfn}
(i) A {\em quasimodular polynomial for $\G$ of weight $\gl$ and
degree at most $m$\/} is an element of $\mF_m [X]$ that is
$\G$-invariant with respect to the right $\G$-action in
\eqref{hE:js}.  We denote by $QP^m_{\gl}(\G,\chi)$ the space of all
quasimodular polynomials for $\G$ of weight $\gl$ and degree at most $m$
with character $\chi$.

(ii) A formal power series $\Phi (z,X)$ belonging
to $\mF [[X]]$ is a {\it Jacobi-like form for $\G$ of weight
$\gl$\/} with character $\chi$ if it satisfies
\[
(\Phi \mid^J_\gl \g) (z,X) = \chi(\g) \Phi (z,X)
\]
for all $z \in \mcal H$ and $\g \in \G$, where $\mid^J_\gl$ is as
in \eqref{hE:yt}.
\end{dfn}

We denote by $\mJ_\gl (\G, \chi)$ the space of all Jacobi-like
forms for $\G$ of weight $\gl$ with character $\chi$, and set
\[ \mJ_\gl (\G, \chi)_\gd = \mJ_\gl (\G, \chi) \cap \mF [[X]]_\gd \]
for each nonnegative integer $\gd$.

We now extend the notion of Hecke operators to the half-integral cases.  We
first recall that $GL^+ (2, \mbb R)$ acts on $\mcal H$ by linear fractional
transformations, and set
\[ \mG =\{ (\ga, \det(\ga)^{-1/2} \fj (\ga, z)) \mid \ga \in GL^+ (2,
\bR) \} .\]
Then $\mG$ is a group with respect to the multiplication given by
\[ (\ga,  \det(\ga)^{-1/2} \fj(\ga, z)) \cdot
(\gb, \det(\gb)^{-1/2} \fj(\gb,z))=(\ga \gb, \det(\ga \gb)^{-1/2}\fj(\ga,
\gb(z))\fj(\gb, z)) .\]
 We shall write
$\tilde{\ga} = (\ga, \det(\ga)^{-1/2} \fj (\ga, z)) \in \mG$.

As in the case of $GL^+ (2, \bR)$, two subgroups $\G_1$ and $\G_2$ of $\mG$
are commensurable, or $\G_1 \sim \G_2$, if $\G_1 \cap \G_2$ has finite
index in both $\G_1$ and $\G_2$.  Given a subgroup $\gD$ of $\mG$, the
subgroup
\[ \wt{\gD} = \{ \tilde{g} \in \mG \mid g \gD g^{-1} \sim \gD \} \subset \mG \]
is the commensurator of $\gD$.  If $\mG_1= \{\tilde{\ga}\in \mG
\mid \det \ga =1\}$ and if $\G$ is a discrete subgroup of $\mG_1$,
then the double coset $\G \ga \G$ with $\tilde{\ga} \in \wt{\G}$
has a decomposition of the form
\begin{equation} \label{hE:xm}
\G \ga \G = \coprod_{i=1}^s \G \ga_i
\end{equation}
for some elements $\ga_i \in \mG$ with $1 \leq i \leq s$.

Given $\gl$ with $2\gl \in \bZ$ odd, we extend the actions of $SL(2,\bR)$
in \eqref{E:xt}, \eqref{E:yt} and \eqref{E:js} to those of $\mG$ by
setting
\[
(f \mid_\gl \ga) (z) = \det(\ga)^{\gl} \fj (\ga, z)^{-2\gl} f(\ga
z)
\]
\begin{equation} \label{hE:sn}
(\Phi \mid^J_\gl \ga) (z,X) = (\det \xi)^{\gl} \fj (\ga,z)^{-2\gl}
e^{-\fk(\ga, z)} \Phi(\ga z, (\det \ga) \fj(\ga,z)^{-2}X),
\end{equation}
\[
(F \md_\gl \ga) (z, X) = (\det \ga)^{ {\gl} } \fj (\ga,z)^{-2\gl}
F (\ga z, (\det \ga)^{-1} \fj(\ga, z)^2( X - \fk (\ga, z)))
\]
for all $z \in \mH$, $\tilde{\ga} \in \mG$, $f \in \mF$, $\Phi
(z,X) \in \mF [[X]]$ and $F (z,X) \in \mF_m [X]$.  Then we have
\begin{equation} \label{hE:2hp}
(f \mid_\gl \ga) \mid_\gl \ga' = f \mid_\gl (\ga\ga'), \quad (\Phi
\mid^J_\gl \ga) \mid^J_\gl \ga' = \Phi \mid^J_\gl (\ga\ga'), \quad
(F \md_\gl \ga) \md_\gl \ga' = F \md_\gl (\ga\ga')
\end{equation}
for all $\tilde{\alpha}, \tilde{\alpha'}, $  and therefore $\mG$
acts on $\mF$, $\mF [[X]]$ and $\mF_m [X]$ on the right.

Let $\tilde{\ga} \in \wt{\G} \subset \mG$ be an element whose
double coset is as in \eqref{hE:xm}.  Then the associated Hecke
operator
\begin{equation} \label{hE:fs}
T_\gl: M_\gl (\G, \chi) \to M_\gl (\G, \chi)
\end{equation}
is given as usual by
\[ (T_\gl (\tilde{\ga}) f) (z) = \det(\ga)^{- {\gl} -1} \sum_{i=1}^s
(\chi(\ga_i)  f \mid_\gl \ga_i) (z) \] for all $f \in M_\gl (\G,
\chi)$ and $z \in \mH$. Similarly, given a Jacobi-like form $\Phi
(z,X) \in \mJ_\gl (\G, \chi)$ and a quasimodular polynomial $F
(z,X) \in QP^m_\gl (\G, \chi)$, we set
\begin{align} \label{hE:jh}
(T^J_\gl (\tilde{\ga}) \Phi) (z,X) &= \det(\ga)^{- {\gl}
-1}\sum_{i=1}^s (\Phi \mid^J_\gl \ga_i) (z,
X),\\
\label{hE:ju} (T^P_{\gl} (\tilde{\ga}) F) (z,X) &= \det(\ga)^{-
{\gl} -1}\sum_{i=1}^s (F \md_{\gl} \ga_i) (z, X)
\end{align}
for all $z \in \mcal H$.

\begin{prop}
For each $\tilde{\ga} \in \wt{\G}$ the formal power series
$(T^J_\gl (\tilde{\ga}) \Phi) (z,X)$ and the polynomial $(T^P_\gl
(\tilde{\ga}) F) (z,X)$ given by \eqref{hE:jh} and \eqref{hE:ju},
respectively, are independent of the choice of the coset
representatives $\ga_1, \ldots, \ga_s$, and the maps $\Phi \mapsto
T^J_\gl (\tilde{\ga}) \Phi$ and $F \mapsto T^P_\gl (\tilde{\ga})
F$ determine linear endomorphisms
\begin{equation} \label{E:8m}
T^J_\gl (\tilde{\ga}): \mJ_\gl (\G, \chi) \to \mJ_\gl(\G, \chi)
,\quad T^P_\gl (\tilde{\ga}): QP^m_\gl (\G, \chi) \to QP^m_\gl
(\G, \chi) .
\end{equation}
\end{prop}

\begin{proof}
This can be proved as in the case of the usual Hecke operators for modular
forms.
\end{proof}

The endomorphisms $T^J_\gl (\tilde{\ga})$ and $T^P_\gl
(\tilde{\ga})$ are Hecke operators on $\mJ_\gl (\G, \chi)$ and
$QP^m_\gl (\G, \chi)$, respectively, for half-integral Jacobi-like
forms and quasimodular polynomials with character.

\section{\bf{Shimura Correspondences}}

In the classical theory of modular forms, a Shimura correspondence provides
a map from from half integral weight modular forms to integral weight
modular forms that is Hecke equivariant.  In this section we consider
similar maps for quasimodular forms.

First, we review Shimura's construction of a Hecke-equivariant map from
half integral weight cusp forms to integral weight cusp forms (see
\cite{Sh73}).  Let $\chi$ be a Dirichlet character, so that the associated
$L$-function is given by
\[ L(s,\chi) =\sum_{m=1}^{\infty}\frac{\chi (m)}{m^{s}} .\]
Given a half integer $\gl$ with $2\gl$ odd, let $S_\gl (\G_0 (N), \chi)$ be
the space of cusp forms for $\G_0 (N)$ of weight $\gl$ with character
$\chi$.  We consider an element $f \in S_\gl (\G_0 (N), \chi)$, so that it
satisfies
\[ (f \mid_{\gl} \g) (z) = \fj (\g,z)^{-2\gl} f(\g z) =\chi(d) f(z) \]
for all $\g \in \G_0(N)$ and $z \in \mH$, where $d$ is the
$(2,2)$-entry of $\g$.  We recall that the Hecke operators
$T^{N}_{ \gl,\chi}(\tilde{\ga})$ on $S_{\gl}(\G_0(4N), \chi)$ are
given by
\[ T^{N}_{ \gl,\chi}(\tilde{\ga}) f = (\det \ga)^{-  \gl -1} \sum_{\nu=1}^r
\chi(\ga_{\nu}) (f|_{\gl}\ga_{\nu}) . \] The following is the main
result obtained by Shimura in \cite{Sh73}.

\begin{thm} \label{T:5h}
Suppose that $g(z) =\sum_{n=1}^{\infty}b(n)q^n \in S_{k+ 1/2}(\G_0(4N),
\chi)$ with $q = e^{2\pi i z}$ is a half-integral weight cusp form with
$k\geq 1.$ Let $t$ be a positive square-free integer, and define the
Dirichlet character $\psi_t$ by $\psi_t (n) =\chi (n) (\frac {-1}{n})^{k}
(\frac {t}{n})$.  If the complex numbers $A_t(n)$ are defined by
\[ \sum_{n=1}^{\infty}\frac{A_t(n)}{n^s} =L(s-k+1, \psi_t)
\sum_{n=1}^{\infty}\frac{b(tn^2)}{n^s}, \]
then the function
\[ S_{t,k}(g(z)) =\sum_{n=1}^{\infty} A_t(n) q^n \]
is a weight $2k$ modular form belonging to $M_{2k} (\G_0 (2N), \chi^2)$.
If $k\geq 2$ then $S_{t,k}(g(z))$ is a cusp form.  Furthermore, if $k=1,$
then $S_{t, 1}(g(z))$ is a cusp form if $g(z)$ is in the orthogonal
complement of the subspace of $S_{3/2}(\G_0(4N), \chi)$
spanned by single variable theta functions.
\end{thm}

From Theorem \ref{T:5h} we obtain the Hecke equivariant Shimura map
\[ \sh_{\gl, t, \chi} : S_{\gl}(\G_0(4N), \chi) \to M_{2\gl-1}( \G_0(2N), \chi^2) \]
defined by
\[
\sh_{\gl,t,\chi} (f(z)) = \sum^\infty_{n =1} A_{\gl,t}(n)q^n
\]
for
\[ f(z) = \sum^\infty_{n =1} a_{\gl}(f, n)q^n ,\]
where
\[
A_{\gl,t}(n)= L(s-\gl+ 5/2, \chi_t) a_{\gl}(f, n^2 t)
\]
for all $n \geq 1$.

In order to discuss the quasimodular analog of the Shimura map, we consider
a quasimodular polynomial $F (z,X) \in QP_{\gl+2m}^{m -r} (\G_0
(4N), \chi)$ of the form
\begin{equation} \label{E:64}
F (z,X) = \sum_{u=0}^{m-r} f_u (z) X^u
\end{equation}
for $z \in \mH$, and set
\begin{equation} \label{E:28}
(\qs^{m,m', r}_{\gl, t, \chi} F) (z,X) = \sum_{\ell=0}^{m'-r} \frac{\BG
  (m'-r+1)\BG (2\gl +4r) \sh^{(m'-r-\ell)}_{\gl+2r,t,\chi} (f_{m-r}) (z)}
   {\BG(\ell+1)
\BG(m'-r-\ell+1) \BG( m'-\ell+2\gl +3r)}X^{\ell} ,
\end{equation}
which is an element of $\mF_{m'-r} [X]$ with $m' \geq 0$.

\begin{prop} \label{P:sh}
The formula \eqref{E:28} determines the Hecke equivariant linear map
\[ \qs^{m,m', r}_{\gl, t, \chi}:
QP_{\gl+2m}^{m -r} (\G_0(4N), \chi ) \to QP_{2(\gl+
 m'+r)-1  }^{m'-r } (\G_0(2N), \chi^2) \]
for $m,m' \geq 0$ and $r \leq \min \{ m, m' \}$.
\end{prop}

\begin{proof}
We first note that that  the map $\fS_m$ in \eqref{E:hh} can be extended to
the half integral case, so that
\[ \fS_m (QP_{\gl}^m (\G,\chi)) \subset M _{\gl-2m} (\G,\chi) .\]
The lifting formula \eqref{E:ak} can also be modified so that for each $h
\in M_\mu (\G, \chi)$ with $\mu \in \frac 12 \bZ$ the polynomial
\begin{equation} \label{E:93}
(\Xi^{2\mu}_m h) (z,X) = \sum^m_{\ell=0} \frac {\BG (m+1) \BG
(2\mu) h^{(m-\ell)} (z)} {\BG (\ell +1) \BG (m-\ell +1) \BG
(m-\ell + 2\mu) } X^\ell
\end{equation}
is a quasimodular polynomial belonging to $QP^m_{\mu +2m} (\G, \chi)$ and
satisfies
\[ \fS_m (\Xi^{2\mu}_m h) = h .\]
We now consider a quasimodular polynomial $F (z,X) \in QP_{\gl+2m}^{m -r} (\G_0
(4N), \chi)$ given by \eqref{E:64}, so that
\[ \fS_{m-r}(F) \in M_{\gl + 2r} (\G_0 (4N), \chi) ,\]
\[ \sh_{\gl+r,t, \chi} (\fS_{m-r} F) \in M_{2\gl +4r -1} (\G_0(2N), \chi^2) .\]
Thus, using \eqref{E:28} and \eqref{E:93}, we see that
\begin{align} \label{E:23}
\Xi^{2\gl+ 4r -1}_{m'-r}& (\sh_{\gl+r,t, \chi} (\fS_{m-r} F))\\
&= \sum_{\ell=0}^{m'-r} \frac{\BG   (m'-r+1)\BG (2\gl +4r)
  \sh^{(m'-r-\ell)}_{\gl+2r,t,\chi} (f_{m-r}))} {\BG(\ell+1) \BG(m'-r-\ell+1) \BG(
  m'-\ell+2\gl + 3r)} X^{\ell} \notag\\
&= \qs^{m,m', r}_{\gl, t, \chi} (F); \notag
\end{align}
hence it follows that
\[ \qs^{m,m', r}_{\gl, t, \chi} (F) \in QP_{2(\gl+  m'+r)-1  }^{m'-r }
(\G_0(2N), \chi^2) .\]
On the other hand, the maps $\fS_{m-r}$ and $\Xi^{2\gl+ 4r -1}_{m'-r}$ are
Hecke equivariant by the commutativity of the diagrams \eqref{E:11} and
\eqref{E:54}, respectively; hence  the Hecke equivariance of $\qs^{m,m',
r}_{\gl, t, \chi}$ follows from the same property for the Shimura map
$\sh_{\gl+r,t, \chi}$ of modular forms.
\end{proof}

From \eqref{E:23} we obtain the relation
\[
\qs^{m,m', r}_{\gl, t, \chi} = \Xi^{2\gl+ 4r -1}_{m'-r} \circ
\sh_{\gl+r,t, \chi} \circ \fS_{m-r}
\]
and therefore the Hecke equivariant commutative diagram
\[
\begin{CD}
 QP_{ \gl+2m}^{m -r}(\G_0(4N), \chi ) @> \qs^{m,m', r}_{\gl, t, \chi} >>
 QP_{2(\gl+r + m' )-1}^{m'-r} (\G_0(2N), \chi')\\
@  V \fS_{m-r} V V   @ AA \Xi^{2\gl+ 4r -1}_{m'-r} A\\
S_{\gl+2r  }(\G_0(4N), \chi ) @> {Sh_{\gl+r, \chi}} >> S_{2(\gl
 +2r)-1}(\G_0( 2N), \chi^2)
\end{CD}
\]
for each $r \in \{ 0, \ldots, m\}$.

Assuming that $m' \geq m$, we now consider a quasimodular polynomial
\[ G (z,X) \in QP_{\gl+2m}^m (\G_0(4N), \chi) ,\]
and define for each $r \in \{ 0, 1, \ldots , m \}$ the quasimodular
polynomial
\[ G_r (z,X) \in QP_{\gl+2m }^{m-r} (\G_0(4N), \chi ) \]
by setting
\[ G_0 = G, \quad G_{\ell +1} = G_{\ell} -(m-\ell)! (\Pi_{m-\ell}  \circ
\wh{\Xi}_0 \circ \fS_{m-\ell}) G_{\ell } \in QP_{\gl+2m
}^{m-\ell}(\G_0(4N), \chi) \] for $0 \leq \ell \leq m-1$, where
$\wh{\Xi}_0$ is as in \eqref{E:m2}.  Then, if we set
\[ \bigoplus_{r=0}^m \qs^{m,m', r}_{\gl, t, \chi} (G)
=( \qs^{m,m', 0}_{\gl, t, \chi} G_0,  \qs^{m,m', 1}_{\gl, t, \chi}
G_1 , \ldots , \qs^{m,m', m}_{\gl, t, \chi} G_m ) \] we obtain the complex linear
map
\[
\bigoplus_{r=0}^m \qs^{m,m', r}_{\gl, t, \chi}: QP_{
\gl+2m}^{m}(\G_0(4N), \chi )   \to \bigoplus_{r=0}^m
QP_{2(\gl+r+m')-1  }^{m'-r} (\G_0(2N), \chi^2) ,
\]
which is Hecke equivariant.

\section{\bf{Shintani Liftings}}

It was Shintani (cf.\ \cite{Sh75}) who constructed Hecke
equivariant maps from integral weight  $\lambda $  cusp forms to
half-integral weight cusp forms, which may be regarded as inverses
of Shimura maps.  In this section we study the quasimodular
version of Shintani maps.

We first review the construction of Shintani maps for modular forms.  Let
$\fQ$ be the space of integral indefinite binary quadratic forms of the
form
\[ Q =Q(X,Y) =aX^2+bXY+cY^2 =[a,b,c] \]
with $\dsc(Q)=b^2-4ac>0$.  Given a positive integer $M$, we set
\[ \fQ_M = \{Q(X,Y)=[a,b,c] \in \fQ \mid (a, M) =1,\; b \equiv
c\equiv 0 \!\!\!\!\pmod{M} \} \]
if $M$ is odd, and
\[ \fQ_M = \{ Q(X,Y)=[a,b,c] \in \fQ  \mid (a, M) =1,\; b\equiv 0
 \!\!\!\!\pmod{2M},\; c\equiv 0  \!\!\!\!\pmod{M}\} \]
if $M$ is even.  Then the congruence subgroup $\G_0(M)$ acts on $\fQ_M$ on
the left by
\[ (\g \cdot Q) (X,Y) =Q((X,Y) \g^{-1}) \]
for all $\g \in \G_0(M)$ and $Q\in \fQ$.  Following \cite{Sh75}, to each
integral indefinite binary form $Q \in \fQ_M$ we associate a pair of points
$\go_Q, {\go'}_Q  \in \bP^1(\bR)=\bR \cup \{ i\infty \}$ given by
\[
(\go_Q , {\go'}_{Q}  )=
\begin{cases}
\Bigl( \frac {b+\sqrt{\dsc (Q)}}{2c}, \frac {b-\sqrt{\dsc (Q)}}{2c} \Bigr)
& \text{if $c \neq 0$;}\\
(i \infty, \frac{a}{b}) & \text{$c=0$ and $b > 0$;}\\
(\frac{a}{b}, i\infty) & \text{$c=0$ and $b < 0$.}
\end{cases}
\]
Given $Q \in \fQ$, we denote by $\g_Q$ the unique generator of the
stabilizer of $Q \in \fQ_M $ in the congruence group $\G_0 (M)$.  We then
consider the path $C_Q$ in $\mH$ defined by
\[ C_Q =
\begin{cases}
\ov {(\go_Q, \go'_{Q})} & \text{if $\dsc (Q)$ is a perfect square;}\\
\ov {(\go, \g_Q \go)} & \text{otherwise,}
\end{cases}\]
where $\go $ is an arbitrary point in $\bP^1 (\bQ)$ and $\ov{(\cdot,
\cdot)}$ denotes the oriented geodesic path joining the given pair of
points.

We write
\[ \chi (Q) = \chi (a) \]
if $Q=[a,b,c]\in \fQ_M$.  Given $f\in S_{2\gl-1}(\G_0(M),
\chi^2)$, we also set
\[
\gT_{\gl,\chi}(f) =
\begin{cases}
\sum_{Q\in \fQ / \G_0(M)} I_{\gl,\chi} q^{\frac{\dsc (Q)}{M}} & \text{if $M$
is odd;}\\
\sum_{Q\in \fQ / \G_0(M)} I_{\gl,\chi} q^{\frac{\dsc (Q)}{4M}}
& \text{if $M$ is even,}
\end{cases}
\]
and
\[ I_{\gl,\chi}(f, Q) = \chi(Q) \int_{C_Q} f(\tau) Q(1,
-\tau)^{\gl-\frac{3}{2}}  d\tau. \] We denote by
$S_{\gl}(\G_0(4M), \chi')$ the space of cups forms of level $4M,$
half integral weight $\gl,$ and Nebentype character $\chi'$.  The
following result was obtained by Shintani.

If $\chi$ is a Dirichlet character $\chi$ defined modulo $M$, we
define the associated Nebentype character $\chi'$ modulo $M$ by
\[ \chi'(d) =\chi(d) \left(\frac{(-1)^{2\lambda+1}M}{d}\right) \]
for all $d\in (\bZ  / 4M \bZ)^{\times}$.

\begin{thm}
Let $\chi$ be a Dirichlet character defined modulo $M$.  Then for each
$f\in S_{2\gl-1}(\G_0(M), \chi^2),$ the series $\gT_{\gl,\chi}(f)$ is the
$q $-expansion of a half-integral weight cusp form in $S_{\gl}(\G_0(4M),
\chi')$.  Moreover, the map
\[ \gT_{\gl,\chi}: S_{2\gl-1}(\G_0(M), \chi^2) \to
S_{\gl} (\G_0(4M), \chi') \]
is a Hecke-equivalent $\bC$-linear map.
\end{thm}

\begin{proof}
See \cite{Sh75}.
\end{proof}

To introduce a quasimodular version of Shintani's result, we consider
a quasimodular polynomial $F (z,X) \in QP_{2\gl-1+2m}^{m -r} (\G_0(4M), \chi^2)$
of the form
\[
F (z,X) = \sum_{u=0}^{m-r} f_u (z) X^u
\]
for $z \in \mH$, and set
\begin{equation} \label{E:2u}
(\qt^{m,m', r}_{\gl, \chi} F) (z,X) = \sum_{\ell=0}^{m'-r} \frac {\BG
  (m'-r+1) \BG (2 \gl++2r )\gT_{\gl+r,\chi}^{(m'-r-\ell)} (f_{m-r}) (z)}
{\BG(\ell+1)\BG(m'-r-\ell+1) \BG( m'-\ell+2\gl+r)}X^{\ell} ,
\end{equation}
which is an element of $\mF_{m'-r} [X]$ with $m' \geq 0$.

\begin{prop} \label{P:sn}
The formula \eqref{E:2u} determines the Hecke equivariant linear map
\[ \qt^{m,m', r}_{\gl, \chi}: QP_{2\gl-1+2m}^{m -r}(\G_0(M), \chi^2) \to
QP_{2m'+\gl-r   }^{m'-r } (\G_0(4M), \chi') \] for $m, m' \geq$
and $0 \leq r \leq \min \{ m,m'\}$.
\end{prop}

\begin{proof}
Given a quasimodular polynomial $F (z,X) \in QP_{2\gl-1+2m}^{m -r}
(\G_0(M), \chi^2)$, we have
\[ \fS_{m-r} F \in S_{2(\gl +r) -1} (\G_0(M), \chi^2) ,\]
\[ \gT_{\gl+r, \chi} (\fS_{m-r} F) \in S_{\gl +r} (\G_0 (4M), \chi') .\]
Thus, using \eqref{E:93}and \eqref{E:2u}, we see that
\begin{align} \label{E:4u}
\Xi^{2(\gl +r)}_{m'-r}& (\gT_{\gl+r, \chi} (\fS_{m-r} F))\\
&= \sum_{\ell=0}^{m'-r} \frac {\BG
  (m'-r+1) \BG (2\gl+2r) \gT_{\gl+r,\chi}^{(m'-r-\ell)} (f_{m-r}) (z)}
{\BG(\ell+1)\BG(m'-r-\ell+1) \BG( m'-\ell+2\gl+r)}X^{\ell} \notag\\
&= \qt^{m,m', r}_{\gl, \chi} (F); \notag
\end{align}
hence it follows that
\[ \qt^{m,m', r}_{\gl, \chi} (F) \in QP_{2m'+\gl-r   }^{m'-r }
(\G_0(4M), \chi') .\] On the other hand, the maps $\fS_{m-r}$ and
$\Xi^{2(\gl+  r)}_{m'-r}$ are Hecke equivariant by the
commutativity of the diagrams \eqref{E:11} and \eqref{E:54},
respectively; hence  the Hecke equivariance of $\qt^{m,m',
  r}_{\gl, \chi}$ follows from the same property for the Shintani map
$\gT_{\gl+r, \chi}$ of modular forms.
\end{proof}

From \eqref{E:4u} we obtain the relation
\[
\qt^{m,m', r}_{\gl, \chi} = {\Xi^{2(\gl +r)}_{m'-r}} \circ
\gT_{\gl+r, \chi} \circ \fS_{m-r}
\]
as well as the Hecke equivariant commutative diagram
\[
\begin{CD}
QP_{2\gl-1+2m}^{m -r}(\G_0( M), \chi^2)   @> \qt^{m,m', r}_{\gl,
\chi}:
>> QP_{ 2 m'+\gl-r}^{m'-r} (\G_0(4M), \chi')\\
@  V \fS_{m-r} VV   @ AA \Xi^{2(\gl +r)}_{m'-r} A \\
S_{2(\gl+r)-1 }(\G_0(M), \chi^2) @> {\gT_{\gl+r, \chi}} >>
S_{\gl+r }(\G_0(4M), \chi')
\end{CD}
\]
for each $r \in \{ 0,1, \ldots, m\}$.

We now consider a quasimodular polynomial $G (z,X) \in
QP_{2\gl-1+2m}^{m-r} (\G_0(M), \chi^2)$, and define the associated
quasimodular polynomials $G_r (z,X) \in QP_{2\gl-1+2m}^{m-r}
(\G_0(M), \chi^2)$ with $m' \geq m$ for each $0 \leq r \leq m$ by
\[ G_0 = G, \quad G_{\ell +1} = G_{\ell} -(m-\ell)! (\Pi_{m-\ell}  \circ
\wh{\Xi}_0 \circ \fS_{m-\ell}) G_{\ell } \in QP_{2\gl-1+2m
}^{m-\ell} \] for $0 \leq \ell \leq m-1$.  Then by setting
\[ \bigoplus_{r=0}^m \qt^{m,m', r}_{\gl, \chi} (G)
=( \qt^{m,m', 0}_{\gl, \chi} G_0,  \qt^{m,m', 1}_{\gl, \chi} G_1 ,
\ldots , \qt^{m,m', m
}_{\gl, \chi} G_m ) \] we obtain the
$\bC$-linear map
\[
\bigoplus_{r=0}^m \qt^{m,m', r}_{\gl, \chi}: QP_{ 2\gl-1+2m}^{m}
(\G_0( M), \chi^2) \to \bigoplus_{r=0}^m QP_{ 2m'+\gl-r }^{m'-r}
(\G_0(4M), \chi') ,
\]
which is Hecke equivariant.

\providecommand{\bysame}{\leavevmode\hbox to3em{\hrulefill}\thinspace}
\providecommand{\MR}{\relax\ifhmode\unskip\space\fi MR }
\providecommand{\MRhref}[2]{%
  \href{http://www.ams.org/mathscinet-getitem?mr=#1}{#2}
}
\providecommand{\href}[2]{#2}

\end{document}